\documentclass{amsart}

\usepackage{graphicx}
\usepackage{enumerate}
\usepackage{amsmath,amssymb}
\usepackage{ascmac}
\usepackage{here} 
\usepackage{color}
\usepackage{amsaddr}

\newtheorem{mtheorem}{Theorem}

\newtheorem{mcorollary}[mtheorem]{Corollary}

\newtheorem{theorem}{Theorem}[section]
\newtheorem{lemma}[theorem]{Lemma}

\newtheorem{cor}[theorem]{Corollary}

\theoremstyle{definition}

\newtheorem{remark}[theorem]{Remark}

\numberwithin{equation}{section}
\numberwithin{figure}{section}
\numberwithin{table}{section}

\begin{document}
\title[Moduli of 3-dimensional diffeomorphisms]
{Moduli of 3-dimensional diffeomorphisms with saddle-foci}

\author[Shinobu Hashimoto]{Shinobu Hashimoto} 
\address{Department of Mathematics and Information Sciences,
Tokyo Metropolitan University,
Minami-Ohsawa 1-1, Hachioji, Tokyo 192-0397, Japan.\\
{\tt E-mail address: hashimoto-shinobu@ed.tmu.ac.jp}}

\author
[Shin Kiriki]{Shin Kiriki}
\address{
Department of Mathematics, Tokai University, 4-1-1 Kitakaname, 
Hiratuka Kanagawa, 259-1292, Japan.\\
{\tt E-mail address: kiriki@tokai-u.jp}}

\author
[Teruhiko Soma]{Teruhiko Soma}
\address{Department of Mathematical Sciences,
Tokyo Metropolitan University,
Minami-Ohsawa 1-1, Hachioji, Tokyo 192-0397, Japan.\\
{\tt E-mail address:  tsoma@tmu.ac.jp}}

\thanks{This work was partially supported by JSPS KAKENHI Grant Numbers 17K05283 and 26400093.}

\date{\today}

\begin{abstract}
We consider a space $\mathcal{U}$ of 3-dimensional diffeomorphisms $f$ 
with hyperbolic fixed points $p$ the stable and unstable manifolds of which have quadratic  
tangencies and satisfying some open conditions and such that $Df(p)$ has non-real expanding eigenvalues 
and a real contracting eigenvalue.
The aim of this paper is to study moduli of diffeomorphisms in $\mathcal{U}$.
We show that, for a generic element $f$ of $\mathcal{U}$, 
all the eigenvalues of $Df(p)$ are moduli and the restriction of a conjugacy homeomorphism  
to a local unstable manifold is a 
uniquely determined linear conformal map.
\end{abstract}

\maketitle
The topological classification of structurally unstable diffeomorphisms or vector fields on a manifold $M$ 
is an important subject in the study of dynamical systems.
Palis \cite{pa} suggested that moduli play important roles in such a classification.
For a subspace $\mathcal{N}$ of the diffeomorphism space $\mathrm{Diff}^r(M)$ with $r\geq 1$, 
we say that a value $m(f)$ determined by $f\in \mathcal{N}$ is a \emph{modulus} in $\mathcal{N}$ 
if $m(g)=m(f)$ holds for any $g\in \mathcal{N}$ topologically conjugate to $f$, that is, 
there exists a homeomorphism $h:M\to M$ with $g=h\circ f\circ h^{-1}$.
A modulus for a certain class of vector fields is defined similarly.
We say that a set $\mu_{\mathcal{N}}$ of moduli is \emph{complete} if any $f$, $g\in \mathcal{N}$ 
with $m(f)=m(g)$ for all $m\in \mu_{\mathcal{N}}$ are topologically conjugate. 
For given vector fields $X$, $Y$ on $M$, a candidate for a conjugacy homeomorphism between 
$X$ and $Y$ is found in a usual manner.
In many cases, such a map is well defined in a most part of $M$.
So it remains to show that the map is extended to a homeomorphism on $M$ by using the condition that 
$X$ and $Y$ have the same value for any moduli in $\mu_{\mathcal{N}}$.
On the other hand, in the diffeomorphism case, it would be difficult to find a complete set of 
moduli except for very restricted classes $\mathcal{N}$ in $\mathrm{Diff}^r(M)$.

First we consider the case that $\dim M=2$ and $f_j$ $(j=0,1)$ are elements of $\mathrm{Diff}^r(M)$ $(r\geq 2)$ with two saddle fixed points $p_j$, $q_j$.
Suppose moreover that $W^u(p_j)$ and $W^s(q_j)$ have a quadratic heteroclinic tangency $r_j$ 
and there exists a conjugacy homeomorphism $h$ between $f_1$ and $f_2$ with   
$h(p_0)=p_1$, $h(q_0)=q_1$ and $h(r_0)=r_1$.
Then,  Palis \cite{pa} proved that
$\dfrac{\log |\lambda_0|}{\log |\mu_0|}= \dfrac{\log |\lambda_1|}{\log |\mu_1|}$ holds 
under ordinary conditions, 
where $\lambda_j$ is the contracting eigenvalue of $Df(p_j)$ and 
$\mu_i$ is the expanding eigenvalue of $Df(q_j)$.
In \cite{po}, Posthumus proved that the homoclinic version of Palis' results.
In fact, he proved that, if $f_j$ $(j=0,1)$ has a saddle fixed point $p_j$ with a homoclinic quadratic tangency, 
then $\dfrac{\log |\lambda_0|}{\log |\mu_0|}= \dfrac{\log |\lambda_1|}{\log |\mu_1|}$ holds, 
where $\lambda_j,\mu_j$ are the contracting and expanding eigenvalues of $Df(p_i)$.
Moreover,  he showed that, by using some results of de Melo \cite{dm},  
$\lambda_0=\lambda_1$ and $\mu_0=\mu_1$ hold if $\dfrac{\log |\lambda_0|}{\log |\mu_0|}$ is irrational.
We refer to \cite{dmp, dmvs,pt,mp1,gpvs,ha} and references therein for more results on    
moduli of 2-dimensional diffeomorphisms.
Moduli for 2-dimensional flows with saddle-connections are studied by Palis \cite{pa} and Takens 
\cite{ta} and so on.
In those papers, they present finite sets of moduli which are complete in a 
neighborhood of the saddle connection in $M$.

In this paper, we consider 3-dimensional diffeomorphisms $f$  with a hyperbolic fixed 
point $p$ such that $W^u(p)$ and $W^s(p)$ have a quadratic tangency and $Df(p)$ has non-real 
expanding eigenvalues $re^{\pm\sqrt{-1}\theta}$ with $r>1$ and a contracting eigenvalue $0<\lambda<1$.
Moduli for diffeomorphisms of dimension more than two have been already studied by 
\cite{npt,du2,mp2} and so on.

First we will prove the following theorem.

\begin{mtheorem}\label{thm_A}
Let $M$ be a $3$-manifold and $f_j$ $(j=0,1)$ elements of $\mathrm{Diff}^r(M)$ for some $r\geq 3$ 
which have hyperbolic fixed points $p_j$ and homoclinic quadratic tangencies $q_j$ positively 
associated with $p_j$ 
and satisfy the following conditions.
\begin{itemize}
\setlength{\leftskip}{-18pt}
\item
For $j=0,1$, there exists a neighborhood $U(p_j)$ of $p_j$ in $M$ such that $f_j|_{U(p_j)}$ is linear 
and $Df_j(p_j)$ has non-real eigenvalues $r_je^{\pm \sqrt{-1}\theta_j}$ 
 and a real eigenvalue $\lambda_j$ with $r_j>1$, $\theta_j\neq 0\mod \pi$ and 
$0<\lambda_j<1$. 
\item  
$f_0$ is topologically conjugate to $f_1$ by a homeomorphism $h:M\to M$ with $h(p_0)=p_1$ 
and $h(q_0)=q_1$.
\end{itemize}
Then the following \eqref{A1} and \eqref{A2} hold.
\begin{enumerate}[\rm (1)]
\item\label{A1}
$\dfrac{\log \lambda_0}{\log r_0}=\dfrac{\log \lambda_1}{\log r_1}$.
\item\label{A2} 
Either $\theta_0=\theta_1$ or $\theta_0=-\theta_1\mod 2\pi$.
\end{enumerate}
\end{mtheorem}

Here we say that a homoclinic quadratic tangency $q_0$ is \emph{positively associated} with $p_0$ 
if both $f_0^n(q_0)$ and $f_0^{-n}(\alpha)$ lie in the same component of $U(p_0)\setminus W_{\mathrm{loc}}^u(p_0)$ 
for a sufficiently large $n\in\mathbb{N}$ and any small curve $\alpha$ in $W^s(p_0)$ 
containing $q_0$.
Theorem \ref{thm_A} holds also in the case when $\theta_0=0\mod \pi$ or $-1<\lambda_j<0$ except 
for some rare case, see Remark \ref{r_exception} for details. 

\begin{remark}
Assertion \eqref{A1} of Theorem \ref{thm_A} is implied in  
the case (D) of Theorem 1.1 in \cite[Chapter III]{npt}.
Assertion \eqref{A2} is also proved by Dufraine \cite{du2} under weaker assumptions.
The author used non-spiral curves in $W_{\mathrm{loc}}^u(p)$ emanating from $p$.
On the other hand, we employ unstable bent disks defined in Section \ref{S_front_folding} which are originally 
introduced by Nishizawa \cite{ni}.
By using such disks, we construct a convergent sequence of mutually parallel straight segments in 
$W_{\mathrm{loc}}^u(p)$ which are mapped to straight segments in $W_{\mathrm{loc}}^u(h(p))$ by $h$, 
see Figure \ref{fig_parallel}.
An \emph{advantage} of our proof is that these sequences are applicable to prove our main theorem, Theorem \ref{thm_B} below.
\end{remark}

Results corresponding to Theorem \ref{thm_A} for 3-dimensional flows 
with Shilnikov cycles are obtained by  
Togawa \cite{to}, Carvalho-Rodrigues \cite{cr} and for those with connections of saddle-foci by Bonatti-Dufraine \cite{bd}, Dufraine \cite{du}, Rodrigues \cite{ro} and so on.
See the Section 2 in \cite{ro} for details.
Moreover Carvalho and Rodrigues \cite{cr} present results on moduli of  
3-dimensional flows with Bykov cycles.

\begin{mtheorem}\label{thm_B}
Under the assumptions in Theorem \ref{thm_A}, suppose moreover that $\theta_0/2\pi$ is irrational.
Then the following conditions hold.
\begin{enumerate}[\rm (1)]
\item\label{B1}
$\lambda_0=\lambda_1$ and $r_0=r_1$.
\item\label{B2}
The restriction $h|_{W_{\mathrm{loc}}^u(p_0)}:W_{\mathrm{loc}}^u(p_0)\to W_{\mathrm{loc}}^u(p_1)$ is 
a uniquely determined linear conformal map.
\end{enumerate}
\end{mtheorem}

In contrast to Posthumus' results for 2-dimensional diffeomorphisms, 
the eigenvalues $\lambda_0$ and $r_0$ are proved to be moduli without the assumption that 
$\dfrac{\log \lambda_0}{\log r_0}$ is irrational.

The restriction $h|_{W_{\mathrm{loc}}^u(p_0)}$ is said to be a \emph{linear conformal map} if $h|_{W_{\mathrm{loc}}^u(p_0)}$ 
is represented as 
$h|_{W_{\mathrm{loc}}^u(p_0)}(z)=\rho e^{\sqrt{-1}\omega}z$ $(z\in W_{\mathrm{loc}}^u(p_0))$ for some $\rho\in\mathbb{R}\setminus \{0\}$ and $\omega\in\mathbb{R}$ 
under the natural identification of $W_{\mathrm{loc}}^u(p_0)$, $W_{\mathrm{loc}}^u(p_1)$ with 
neighborhoods of the origin in $\mathbb{C}$ via their linearizing coordinates.

For any $r_j>1$ and $\theta_j\in \mathbb{R}$ $(j=0,1)$, let $\varphi_j:\mathbb{C}\to \mathbb{C}$ be the map 
defined by $\varphi_j(z)=r_je^{\sqrt{-1}\theta_j}z$.
Then there are many choices of conjugacy homeomorphisms on $\mathbb{C}$ for $\varphi_0$ and $\varphi_1$. 
For example, we take two-sided Jordan curves $\Gamma_j$ in $\mathbb{C}$ with $\varphi_j(\Gamma_j)\cap \Gamma_j=\emptyset$ 
and bounding disks in $\mathbb{C}$ containing the origin arbitrarily.
Then there exists a conjugacy homeomorphism $h:\mathbb{C}\to \mathbb{C}$ for $\varphi_0$ and $\varphi_1$ with 
$h(\Gamma_0)=\Gamma_1$.
On the other hand, Theorem \ref{thm_B}\,\eqref{B2} implies that we have severe constraints 
in the choice of conjugacy homeomorphisms for 3-dimensional diffeomorphisms as above.
Intuitively, it says that only a homeomorphism $h$ with $h|_{W^u_{\mathrm{loc}}}(p)$ linear and conformal can be 
a candidate for a conjugacy between $f_0$ and $f_1$.
As an application of the linearity and conformality of $h|_{W^u_{\mathrm{loc}}}(p)$, we will present a new modulus for $f_0$ other than $\theta_0$, $\lambda_0$, $r_0$, see Corollary \ref{c_C} in Section \ref{S_PB}.

\section{Front curves and folding curves}\label{S_front_folding}

For $j=0,1$, let $f_j$ be a diffeomorphism and $q_j$ a quadratic tangency associated with a hyperbolic 
fixed point $p_j$ satisfying the conditions of Theorem \ref{thm_A}. 
We will define in this section front curves in $W^u(p_j)$ and folding curves in $W_{\mathrm{loc}}^u(p_j)$ 
and show in the next section that these curves converge to straight segments which are preserved by any conjugacy 
homeomorphism between $f_0$ and $f_1$.

We set $f_0=f$, $p_0=p$, $q_0=q$, $r_0=r$, $\theta_0=\theta$ and $\lambda_0=\lambda$ for short.
Similarly, let $f_1=f'$, $p_1=p'$, $q_1=q'$, $r_1=r'$, $\theta_1=\theta'$ and $\lambda_1=\lambda'$.
Suppose that $(z,t)=(x,y,t)$ with $z=x+\sqrt{-1}y$ is a coordinate around $p$ with respect to 
which $f$ is linear.
For a small $a>0$, let $D_a(p)$ be the disk $\{z\in\mathbb{C}\,; |z|\leq a\}$.
We may assume that $q$ is contained in the interior of $D_a(p)\times \{0\}\subset W_{\mathrm{loc}}^u(p)$ and $\widehat q=f^N(q)$ is in the interior of the upper half $W_{\mathrm{loc}}^{s+}(p)=\{0\}\times [0,a]$ of $W_{\mathrm{loc}}^s(p)$ 
for some $N\in\mathbb{N}$.
See Figure \ref{fig_1}.
\begin{figure}[hbt]
\centering
\scalebox{0.6}{\includegraphics[clip]{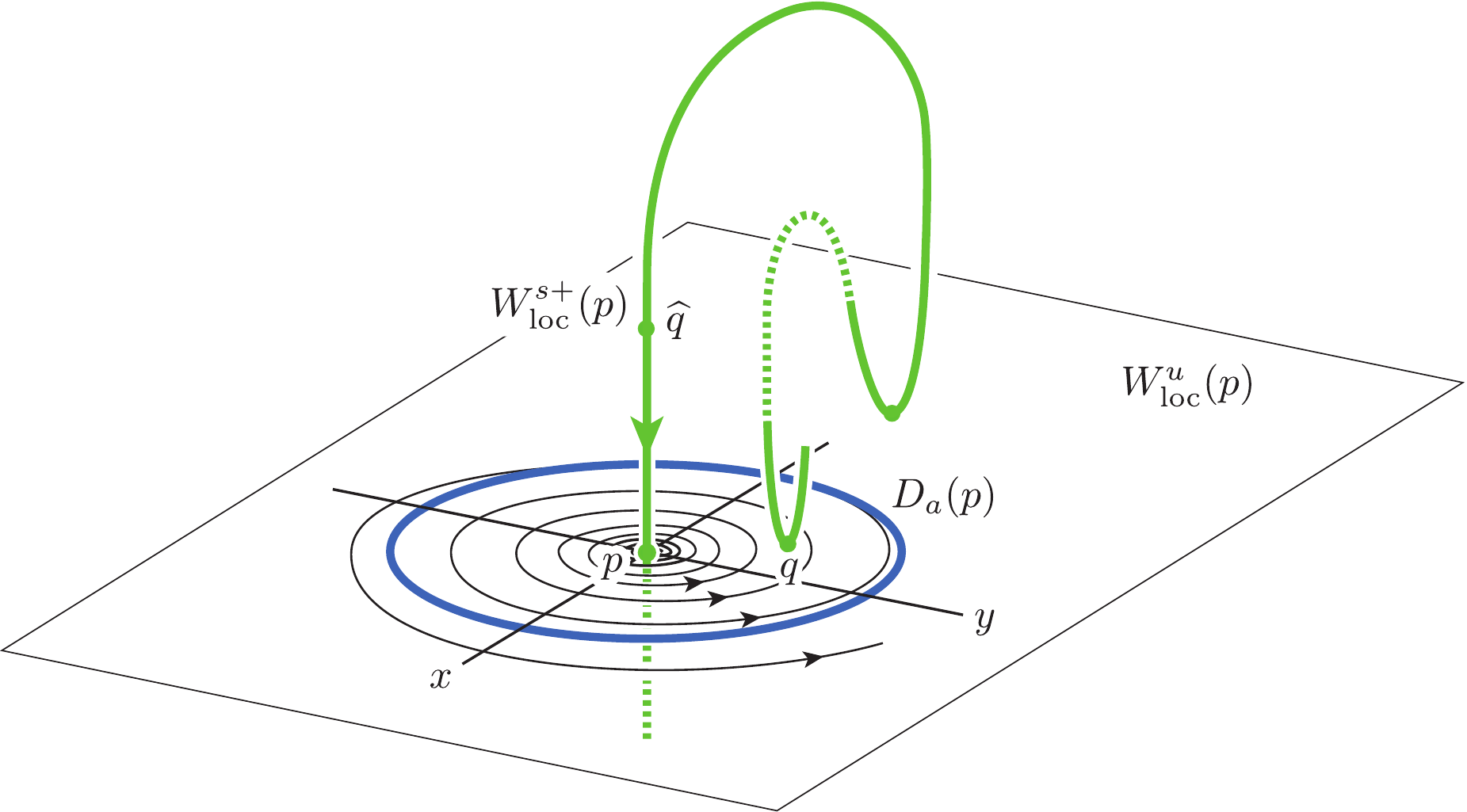}}
\caption{A saddle-focus $p$ and a homoclinic quadratic tangency $q$ in $D_a(p)$.}
\label{fig_1}
\end{figure}
Let $U_a(p)$ be the circular column in the coordinate neighborhood  
defined by $U_a(p)=D_a(p)\times [0,a]$ and $V_{\widehat q}$ a small neighborhood of $\widehat q$ in $U_a(p)$.
Suppose that $U_a(p)$ has the Euclidean metric induced from the linearizing coordinate on $U_a(p)$.
By choosing the coordinate suitably and replacing $\theta$ by $-\theta$ if necessary, 
we may assume that the restriction $f|_{D_a(p)}$ is represented as $re^{\sqrt{-1}\theta}z$ for 
$z\in \mathbb{C}$ with $|z|<a$.
Similarly, one can suppose that $f'|_{D_{a'}(p')}$ is represented as $r'e^{\sqrt{-1}\theta'}z$ for 
some $a'>0$.
The orthogonal projection $\mathrm{pr}:U_a(p)\to D_a(p)$ is defined by $\mathrm{pr}(x,y,t)=(x,y)$.

In this section, we construct an unstable bent disk $\widetilde H_0$ in $W^u(p)\cap U_a(p)$, the front curve $\widetilde\gamma_0$ in $\widetilde H_0$ and the folding curves $\gamma_0$ in $U_a(p)$.
We also define the sequence of unstable bent disks $\widetilde H_m$ in $W^u(p)\cap U_a(p)$ converging to $\widetilde H_0$, 
which will be used in the next section to construct the sequence of front curves converging to $\widetilde\gamma_0$.

\subsection{Construction of unstable bent disks, front curves and folding curves}\label{ss_bent_disk}
We set $\widehat q=(0,t_0)$.
Let $\widetilde H$ be the component of $W^u(p)\cap V_{\widehat q}$ containing $\widehat q$.
One can retake the linearizing coordinate on $\mathbb{C}$ if necessary so that the line in $V_{\widehat q}$ passing through 
$\widehat q$ and parallel to the $x$-axis in $U_a(p)$ meets $\widetilde H$ transversely. 
Then $\widetilde H$ is represented as the graph of a $C^r$-function 
$x=\varphi(y,t)$ with
\begin{equation}\label{eqn_vp}
\varphi(0,t_0)=0,\quad\frac{\partial \varphi}{\partial t}(0,t_0)=0\quad\text{and}\quad\frac{\partial^2 \varphi}{\partial t^2}(0,t_0)\neq 0.
\end{equation}
By the implicit function theorem, there exists a $C^{r-1}$-function $t=\eta(y)$ defined in 
a small neighborhood $V$ of $0$ in the $y$-axis and satisfying $\eta(0)=t_0$ and $\partial \varphi(y,\eta(y))/\partial t=0$.
Then the curve $\widetilde \gamma$ in $V_{\widehat q}$ parametrized by 
$\bigl(\varphi(y,\eta(y)),y,\eta(y)\bigr)$ divides $\widetilde H$ into two components and
$\gamma=\mathrm{pr}(\widetilde\gamma)$ is a $C^{r-1}$-curve embedded in $D_a(p)$.
Let $\widetilde H^+$ (resp.\ $\widetilde H^-$) be the closure of the upper (resp.\ lower) 
component of $\widetilde H\setminus \widetilde\gamma$.
For a sufficiently large $n_0\in\mathbb{N}$, 
the component $\widetilde H_0$ of $f^{n_0}(\widetilde H)\cap U_a(p)$ containing $q_0=f^{n_0}(\widehat q)$ is an  
\emph{unstable bent disk} in 
$U_a(p)$ such that $\partial \widetilde H_0$ is a simple closed $C^r$-curve  
in $\partial_{\mathrm{side}}U_a(p)$, where 
$$\partial_{\mathrm{side}}U_a(p)=\{(x,t)\in\mathbb{C}\times \mathbb{R}\,; |z|= a, 0\leq t< a\}\subset \partial U_a(p).$$
See Figure \ref{fig_H_0}.
We set $\widetilde\gamma_0=f^{n_0}(\widetilde\gamma)\cap \widetilde H_0$, $\widetilde H_0^+=f^{n_0}(\widetilde H^+)\cap \widetilde H_0$, 
$\widetilde H_0^-=f^{n_0}(\widetilde H^-)\cap \widetilde H_0$, $H_0=\mathrm{pr}(\widetilde H_0^+)=\mathrm{pr}(\widetilde H_0^-)$ 
and $\gamma_0=\mathrm{pr}(\widetilde \gamma_0)$. 
Then $\widetilde \gamma_0$ is called the \emph{front curve} of $\widetilde H_0$ and 
$\gamma_0$ is the \emph{folding curve} of $H_0$.
\begin{figure}[hbt]
\centering
\scalebox{0.6}{\includegraphics[clip]{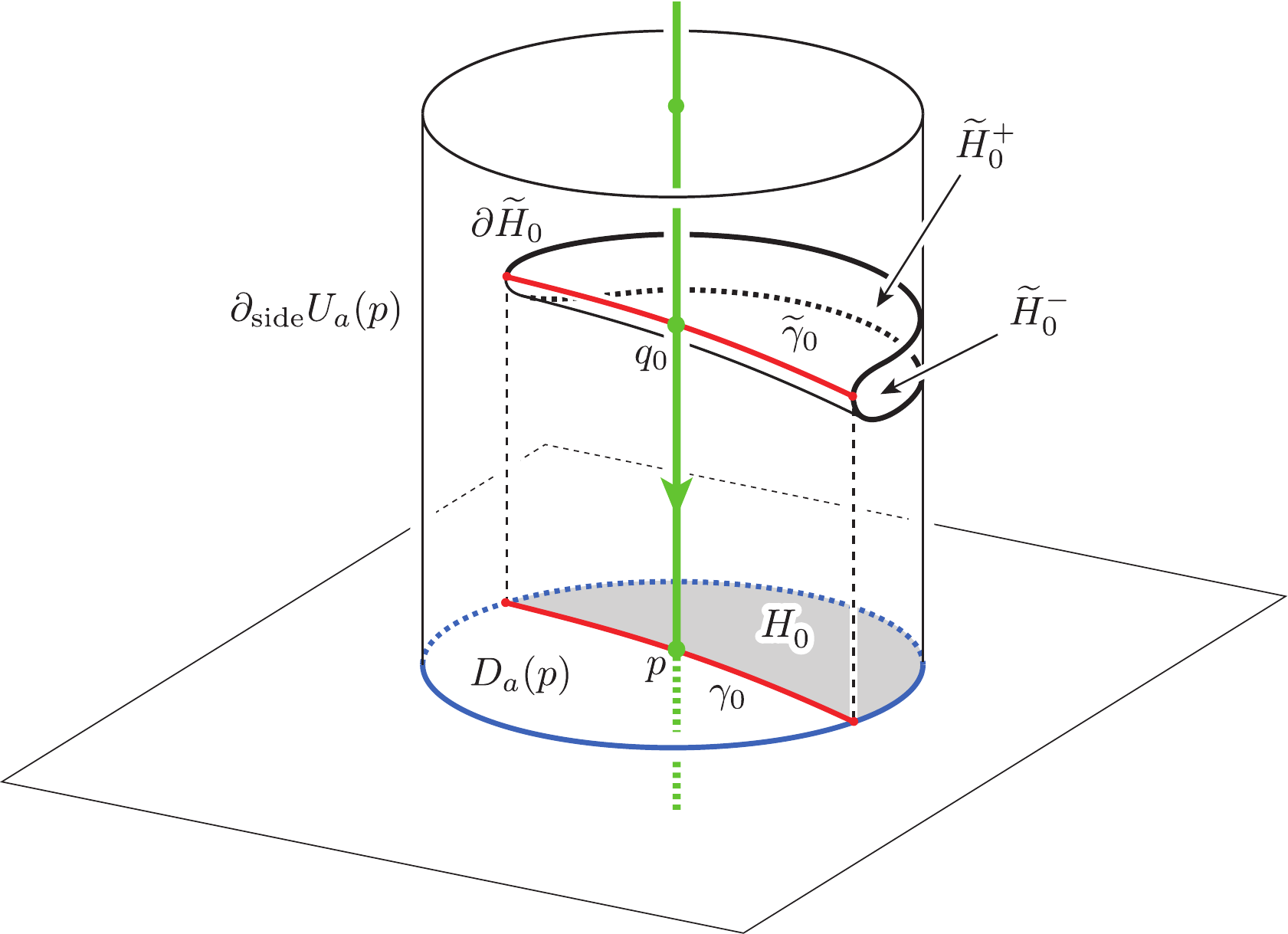}}
\caption{The front curve $\widetilde \gamma_0$ divides $\widetilde H_0$ into the two sheets  
$\widetilde H_0^+$ and $\widetilde H_0^-$.
The folding curve $\gamma_0$ of $H_0$ is the orthogonal image of $\widetilde\gamma_0$.}
\label{fig_H_0}
\end{figure}

We note that Nishizawa \cite{ni} has studied unstable bent disks similar to $\widetilde H_0$ as above in a 
different situation.
In fact, he considered a 3-dimensional diffeomorphism $g$ which has a saddle fixed point $s$ 
such that all the eigenvalues of $Dg(s)$ are real and has a homoclinic quadratic tangency associated 
with $s$.
Here we consider the component $\widetilde H_{0;u}^-$ of $f^u(\widetilde H_0^-)\cap U_a(p)$ containing $f^u(q_0)$ 
for $u\in\mathbb{N}$.
Since the homoclinic tangency $q$ is positively associated with $p$, 
one can show that there exists $\widetilde H_{0;u}^-$ which meets $W^s(p)$ transversely at a point $\widehat z$ 
near $q$ 
by using an argument similar to that in \cite[Lemma 4.4]{ni}.
See Figure \ref{fig_H_u}.
\begin{figure}[hbt]
\centering
\scalebox{0.6}{\includegraphics[clip]{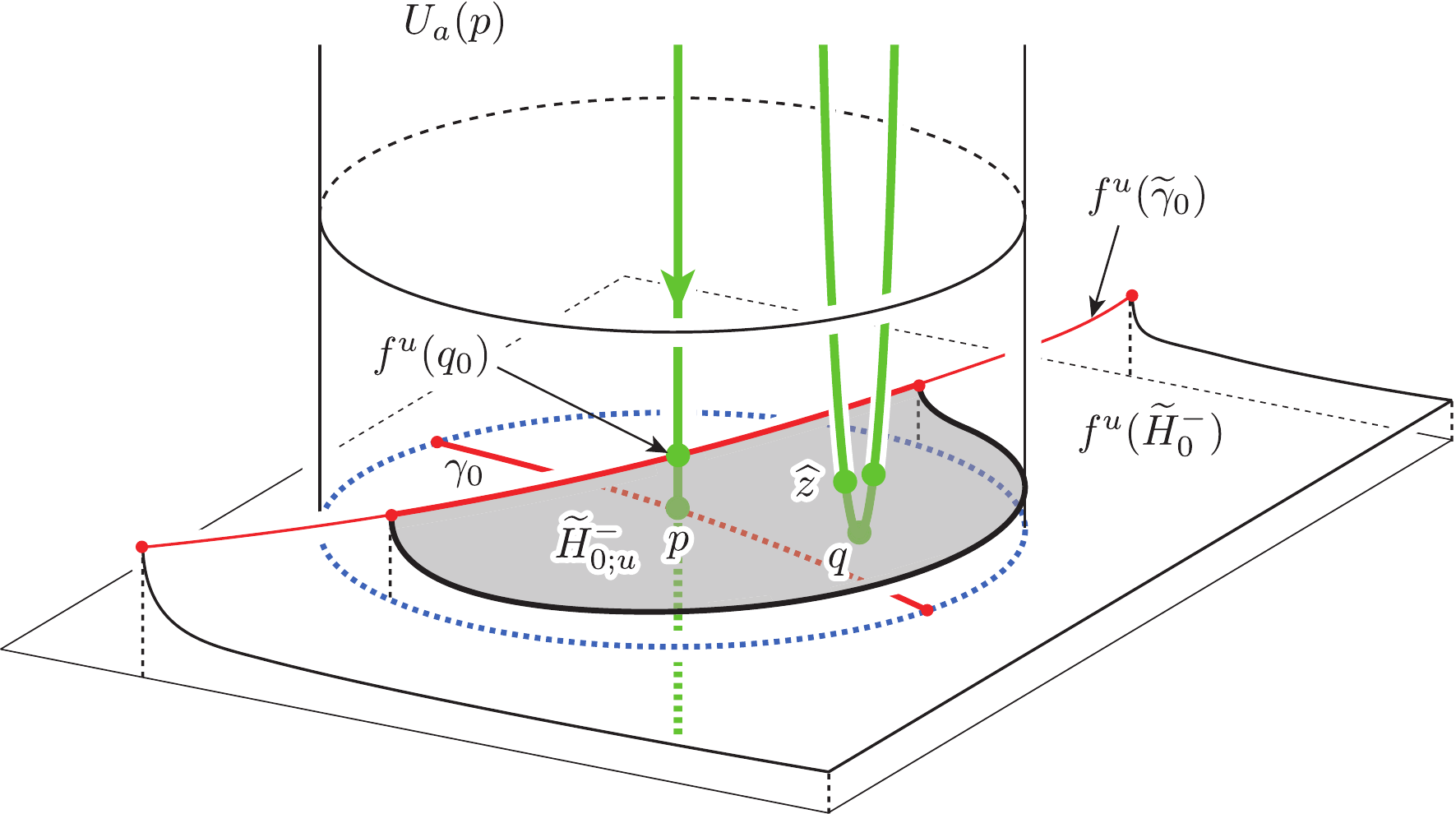}}
\caption{The half disk $\widetilde H_{0;u}^-$ meets $W^s(p)$ transversely at two points near $q$, one of which is $\widehat z$.}
\label{fig_H_u}
\end{figure}
To show the claim, the assumption of $\theta_0\neq 0\mod \pi$ in Theorem \ref{thm_A} is crucial.
In fact, the condition implies that the following property:
\begin{enumerate}[(P)]
\item
There exists an arbitrarily large $u$ such that 
the interior of $H_{0;u}=\mathrm{pr}(\widetilde H_{0;u}^-)$ in $D_a(p)$ contains $q$.
\end{enumerate}

\begin{remark}\label{r_exception}
(1)
We here suppose $\theta=0\mod \pi$.
Even in this case, if $f$ has the property (P), then the component 
of $W^s(p)$ 
containing $q$ and $W^u(p)$ have a homoclinic transverse intersection point.
Then Theorems \ref{thm_A} and \ref{thm_B} will be proved quite similarly.
Since $\theta=0\mod \pi$, all $f^u(\gamma_0)$ are tangent to a unique straight segment $\gamma_\infty$ 
in $D_a(p)$ at $p$.
Thus the property (P) is satisfied if $\gamma_\infty$ does not pass through $q$.

\medskip
\noindent(2)
Even in the case of $-1<\lambda<0$, one can show that $f$ has the property (P) similarly by using $f^2$ instead of $f$ 
if $2\theta\neq 0\mod \pi$.
Moreover, since either  $q$ or $f(q)$ is a homoclinic tangency positively associated with $p$, 
Theorems \ref{thm_A} and \ref{thm_B} hold without the assumption that $q$ is positively associated with 
$p$. 
\end{remark}

\subsection{Construction of convergent sequence of unstable bent disks}
Take $v\in\mathbb{N}$ such that $\widehat z_0=f^v(\widehat z)$ is a point $(0,\widehat t\,)$ contained in $U_a(p)$, 
where $\widehat z$ is the transverse intersection point of $\widetilde H_{0;u}^-$ and $W^s(p)$ given in the previous subsection.
Let $D$ be a small disk in $W^u(p)\cap U_a(p)$ whose interior contains $\widehat z_0$.
The \emph{absolute slope} $\sigma(\boldsymbol{v})$ of a vector $\boldsymbol{v}=(v_1,v_2,v_3)$ in $U_a(p)$ 
with $(v_1,v_2)\neq (0,0)$ 
is given as 
$$\sigma(\boldsymbol{v})=\frac{|v_3|}{\sqrt{v_1^2+v_2^2}}.$$
The \emph{maximum absolute slope} $\sigma(D)$ of $D$ is defined by
$$\sigma(D)=\max\{\sigma(\boldsymbol{v})\,;\, \text{unit vectors $\boldsymbol{v}$ in $U_a(p)$ tangent to $D$}\}.$$
Fix $m_0\in \mathbb{N}$ such that, for any $m\in \mathbb{N}\cup\{0\}$, the component $D_m$ of $f^{m_0+m}(D)\cap U(p)$ 
containing $f^{m_0+m}(\widehat z_0)$ is a properly embedded disk in $U_a(p)$ 
with $\partial D_m\subset \partial_{\mathrm{side}}U_a(p)$.
Note that $D_m$ intersects $W_{\mathrm{loc}}^s(p)$ transversely at $(0,\lambda^mt_0)$, where 
$t_0=\lambda^{m_0}\widehat t$.
See Figure \ref{fig_Dm}.
\begin{figure}[hbt]
\centering
\scalebox{0.6}{\includegraphics[clip]{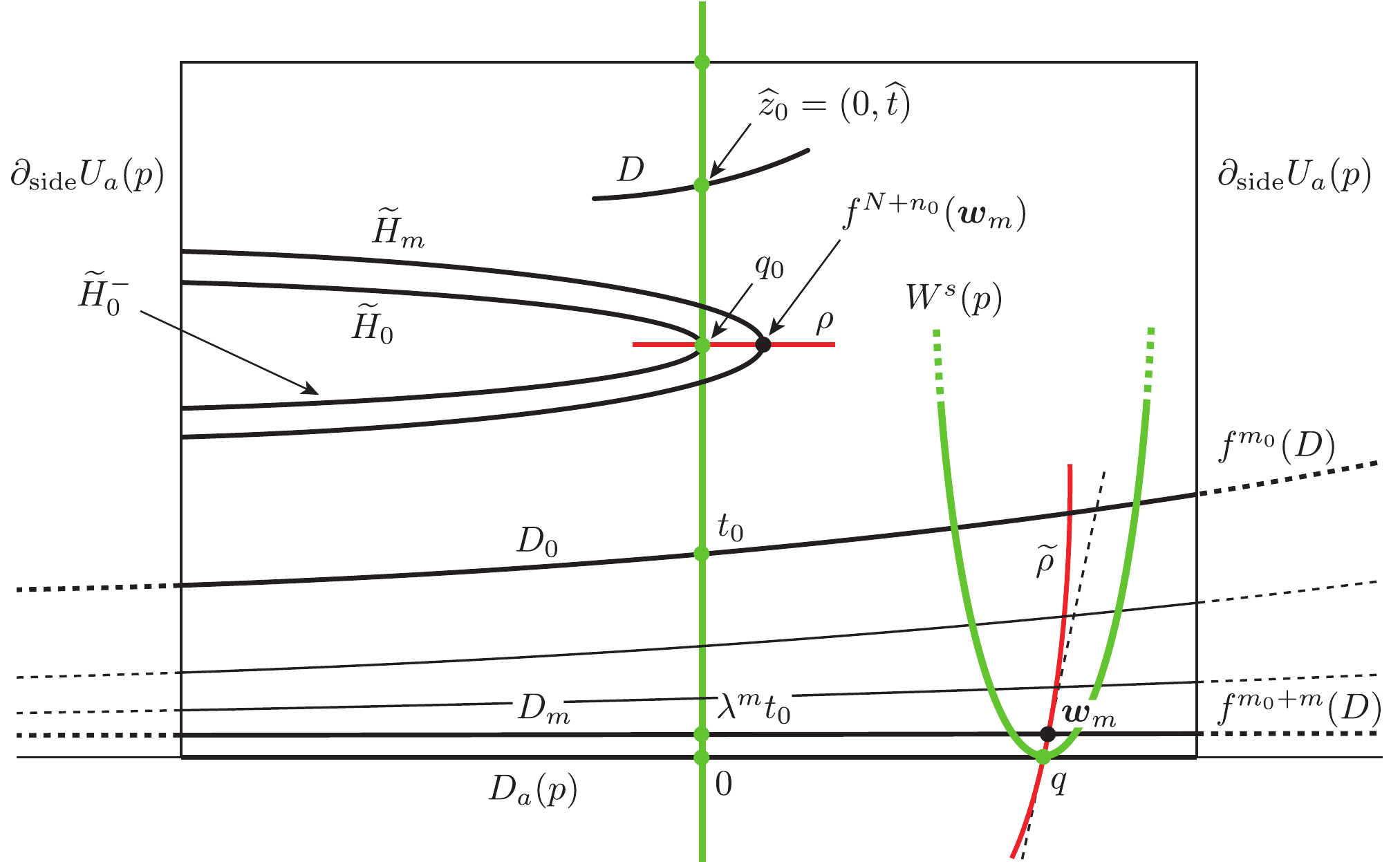}}
\caption{Trip from $\widetilde H_0^-$ to $\widetilde H_m$: $f^{u+v}(\widetilde H_0^-)\supset D$, 
$f^{m_0}(D)\supset D_0$, $f^m(D_0)\supset D_m$ and $f^{N+n_0}(D_m)\supset \widetilde H_m$, where 
$N$, $n_0$ are the positive integers with $f^N(q)=\widetilde q$ and $f^{n_0}(\widetilde q)=q_0$.
The dotted line passing through $q$ represents a straight segment tangent to $\widetilde \rho$ 
at $q$.}
\label{fig_Dm}
\end{figure}
The maximum absolute slope of $D_m$ satisfies 
\begin{equation}\label{eqn_sigma_D}
\sigma(D_m)\leq \sigma_0\lambda^mr^{-m},
\end{equation}
where 
$\sigma_0=\sigma(D)\lambda^{m_0}r^{-m_0}$.
Consider a short straight  segment $\rho$ in $U_a(p)$ meeting $\widetilde H_0$ orthogonally at $q_0$.
Then $\widetilde\rho=f^{-(N+n_0)}(\rho)$ is a $C^r$-curve meeting $D_a(p)$ transversely at $q$, where 
$N$, $n_0$ are the positive integers given as above.
One can choose $m_0\in\mathbb{N}$ so that, for any $m\in\mathbb{N}\cup\{0\}$, $\widetilde\rho$ meets $D_m$ transversely 
at a single point $\boldsymbol{w}_m=(z_m,s_m)$.
Then \eqref{eqn_sigma_D} implies that $|t_0\lambda^m-s_m|\leq \widetilde a\sigma_0\lambda^mr^{-m}$, where 
$\widetilde a=\sup_{m\geq 0}\{|z_m|\}<\infty$.
It follows that $s_m=t_0\lambda^m+O(\lambda^mr^{-m})$.
Since $\widetilde\rho$ has a tangency of order at least two with a straight segment at $q$, 
\begin{equation}\label{eqn_s_m}
\mathrm{dist}(\boldsymbol{w}_m,q)=\widetilde t_0\lambda^m+O(\lambda^mr^{-m})+O(\lambda^{2m})=\widetilde t_0\lambda^m
+o(\lambda^m)
\end{equation}
for some constant $\widetilde t_0>0$.
By the inclination lemma, $D_m$ uniformly $C^r$-converges to $D_a(p)$.
A short curve in $W^s(p)$ containing $q$ as an interior point meets $D_m$ transversely in two points 
for all sufficiently large $m$.
Let $\widetilde H_m$ be the component of $f^{N+n_0}(D_m)\cap U_a(p)$ containing $f^{N+n_0}(\boldsymbol{w}_m)$.
Then $\widetilde H_m$ $C^r$-converges to $\widetilde H_0$ as $m\to \infty$.
By \eqref{eqn_vp}, there exist $C^r$-functions $\varphi_m(y,t)$ $C^r$-converging to $\varphi$ and representing 
$\widetilde H_m$ as the graph of $x=\varphi_m(y,t)$.
Then the front curve $\widetilde \gamma_m$ in $\widetilde H_m$ is defined as the front curve $\widetilde\gamma_0$ in $\widetilde H_0$.
Since $\partial\varphi_m(y,t)/\partial t$ $C^{r-1}$-converges to $\partial \varphi(y,t)/\partial t$, $\widetilde\gamma_m$ 
also $C^{r-1}$-converges to $\widetilde\gamma_0$.
Note that $\widetilde \gamma_m$ divides $\widetilde H_m$ into the upper surface $\widetilde H_m^+$ and the lower surface $\widetilde H_m^-$ with 
$\widetilde \gamma_m=\widetilde H_m^+\cap \widetilde H_m^-$ and
$H_m=\mathrm{pr}(\widetilde H_m)=\mathrm{pr}(\widetilde H_m^+)=\mathrm{pr}(\widetilde H_m^-)$.
The image $\gamma_m=\mathrm{pr}(\widetilde \gamma_m)$ is called the folding curve of $H_m$.

\section{Limit straight segments}\label{S_limit}

A curve $\gamma$ in $D_a(p)$ is called a \emph{straight segment} if $\gamma$ is a segment with respect to the Euclidean metric on $D_a(p)$.
In this section, we will construct a proper straight segment $\gamma_0^\natural$ in $D_a(p)$ with 
$p\not\in\gamma_0^\natural$ which is mapped to a straight segment in $U_{a'}(p')$ by $h$.

\subsection{Sequences of folding curves converging to straight segments}
Let $\alpha$ be an oriented $C^{r-1}$-curve in $D_a(p)$ of bounded length.
Since $r-1\geq 2$, there exists the maximum absolute curvature $\kappa(\alpha)$ of $\alpha$.
If $\alpha$ passes near the center $0$ of $D_a(p)$ and satisfies $\kappa(\alpha)<1/a$, 
then $\alpha$ has a unique point $z(\alpha)$ with 
$\mathrm{dist}(0,z(\alpha))=\mathrm{dist}(0,\alpha)$.
In fact, if $\alpha$ had two points $z_i$ $(i=1,2)$ with $\mathrm{dist}(0,z_i)=\mathrm{dist}(0,\alpha)$, 
then for a point $z_3$ in $\alpha$ with the maximum $\mathrm{dist}(0,z_3)$ between $z_1$ and $z_2$, 
the curvature of $\alpha$ at $z_3$ is not less than $1/\mathrm{dist}(0,z_3)\geq 1/a$, a contradiction.
We denote by $\vartheta(\alpha)\mod 2\pi$ the angle between $\widehat\alpha$ and the positive direction of the $x$-axis at $0$, where 
$\widehat\alpha$ is the oriented curve in $D_a(p)$ obtained from $\alpha$ by the parallel translation 
taking $z(\alpha)$ to $0$.

By \eqref{eqn_s_m}, there exists a constant $\widetilde d_0>0$ such that
\begin{equation}\label{eqn_dcm}
\mathrm{dist}(\widetilde \gamma_m,\text{the $t$-axis})=\widetilde d_0(\widetilde t_0\lambda^m+o(\lambda^m))+o(\lambda^m)
=\widetilde d_0\widetilde t_0\lambda^m+o(\lambda^m).
\end{equation}
Since $\gamma_m$ $C^{r-1}$-converges to $\gamma_0$, $\kappa(\gamma_m)$ also converges to $\kappa(\gamma_0)$ 
as $m\to\infty$.
This shows that
\begin{equation}\label{eqn_kappa}
\sup_m\{\kappa(\gamma_m)\}=\kappa_0<\infty.
\end{equation}
It follows that, for all sufficiently large $m$, 
there exists a unique point $c_m$ of $\gamma_m$ with 
$$\mathrm{dist}(c_m,0)=\mathrm{dist}(\gamma_m,0)=\mathrm{dist}(\widetilde c_m,\text{the $t$-axis})
=\mathrm{dist}(\widetilde \gamma_m,\text{the $t$-axis}),$$
where $\widetilde c_m$ is the point of $\widetilde\gamma_m$ with $\mathrm{pr}(\widetilde c_m)=c_m$.

Fix $w$ with $0<w<a/2$ arbitrarily.
For any $n\in\mathbb{N}$, let  $m(n)$ be the minimum positive integer such that 
$f^n(c_m)$ is contained in $D_w(p)$ for any $m\geq m(n)$.
Then $\lim_{n\to \infty}m(n)=\infty$ holds.
For any $m\geq m(n)$, the component $\widetilde H_{m,n}$ of $f^n(\widetilde H_m)\cap U_a(p)$ containing $\widetilde c_{m,n}=f^n(\widetilde c_m)$ 
is a proper disk in $U_a(p)$ with $\partial \widetilde H_{m,n}\subset \partial_{\mathrm{side}}U_a(p)$.
Then $\widetilde \gamma_{m,n}=f^n(\widetilde \gamma_m)\cap \widetilde H_{m,n}$ is the front curve of $\widetilde H_{m,n}$ 
and $\gamma_{m,n}=\mathrm{pr}(\widetilde \gamma_{m,n})$ is the folding curve of $H_{m,n}=\mathrm{pr}(\widetilde H_{m,n})$.
Then $c_{m,n}=\mathrm{pr}(\widetilde c_{m,n})$ is a unique point of $\gamma_{m,n}$ closest to $0$.
Here we orient $\widetilde\gamma_m=\widetilde \gamma_{m,0}$ so that $\widetilde \gamma_{m,0}$ $C^{r-1}$-converges as oriented curves 
to $\widetilde \gamma_0$ as $m\to\infty$.
Suppose that $\gamma_{m,n}$ has the orientation induced from that on $\widetilde\gamma_{m,0}$ 
via $\mathrm{pr}\circ f^n$.
In particular, it follows that
\begin{equation}\label{eqn_theta}
\lim_{m\to \infty}\vartheta(\gamma_{m,0})=\vartheta(\gamma_0).
\end{equation}
We set $d_{m,n}=\mathrm{dist}(c_{m,n},0)$.
By \eqref{eqn_dcm},
\begin{equation}\label{eqn_dmn}
d_{m,n}=r^n(\widetilde d_0\widetilde t_0\lambda^m+o(\lambda^m)).
\end{equation}
There exist subsequences $\{m_j\}$, $\{n_j\}$ of $\mathbb{N}$ and $w\lambda/2\leq w_0\leq w$ such that 
\begin{equation}\label{eqn_limit_dt}
\lim_{j\to\infty}\widetilde d_0\widetilde t_0\lambda^{m_j}r^{n_j}=w_0.
\end{equation}
If necessary taking subsequences of $\{m_j\}$ and $\{n_j\}$ simultaneously, 
we may also assume that $\vartheta(\gamma_{m_j,n_j})$ has a limit $\theta^\natural$.
Since $f(z)=re^{\sqrt{-1}\theta}z$ on $D_a(p)$, by \eqref{eqn_kappa} we have 
$$\kappa(\gamma_{m_j,n_j})\leq r^{-n_j}\kappa(\gamma_{m_j,0})\leq r^{-n_j}\kappa_0
\to 0\quad\text{as}\quad j\to \infty.$$
Thus the following lemma is obtained immediately.

\begin{lemma}\label{l_gamma_star}
The sequence $\gamma_{m_j,n_j}$ uniformly converges as oriented curves to an oriented straight segment 
$\gamma_0^\natural$ in $D_a(p)$ with $\vartheta(\gamma_0^\natural)=\theta^\natural$ 
and $\mathrm{dist}(\gamma_0^\natural,0)=w_0$.
\end{lemma}

We say that $\gamma_0^\natural$ is the \emph{limit straight segment} of $\gamma_{m_j,n_j}$.

\subsection{Limit straight segments preserved by the conjugacy}
Let $U_{a'}(p')$, $U_{b'}(p')$ be the circular columns defined as $U_a(p)$ for some $0<a'<b'$ 
which are contained in a coordinate neighborhood around $p'$ with respect to which $f'$ is linear.
One can retake $a>0$ and choose such $a'$, $b'$ so that $U_{a'}(p')\subset h(U_a(p))\subset U_{b'}(p')$.

Let $\widetilde H_{m,n}'$ be the component of $h(\widetilde H_{m,n})\cap U_{a'}(p')$ defined as $\widetilde H_{m,n}$  and $\mathrm{pr}(\widetilde H_{m,n}')=H_{m,n}'$.
One can define the front and folding curves $\widetilde \gamma_{m,n}'$, $\gamma_{m,n}'$ in 
$\widetilde H_{m,n}'$ and $H_{m,n}'$ as $\widetilde\gamma_{m,n}$, $\gamma_{m,n}$ in $\widetilde H_{m,n}$ and $H_{m,n}$ respectively.
See Figure \ref{fig_UaUb}.
\begin{figure}[hbt]
\centering
\scalebox{0.6}{\includegraphics[clip]{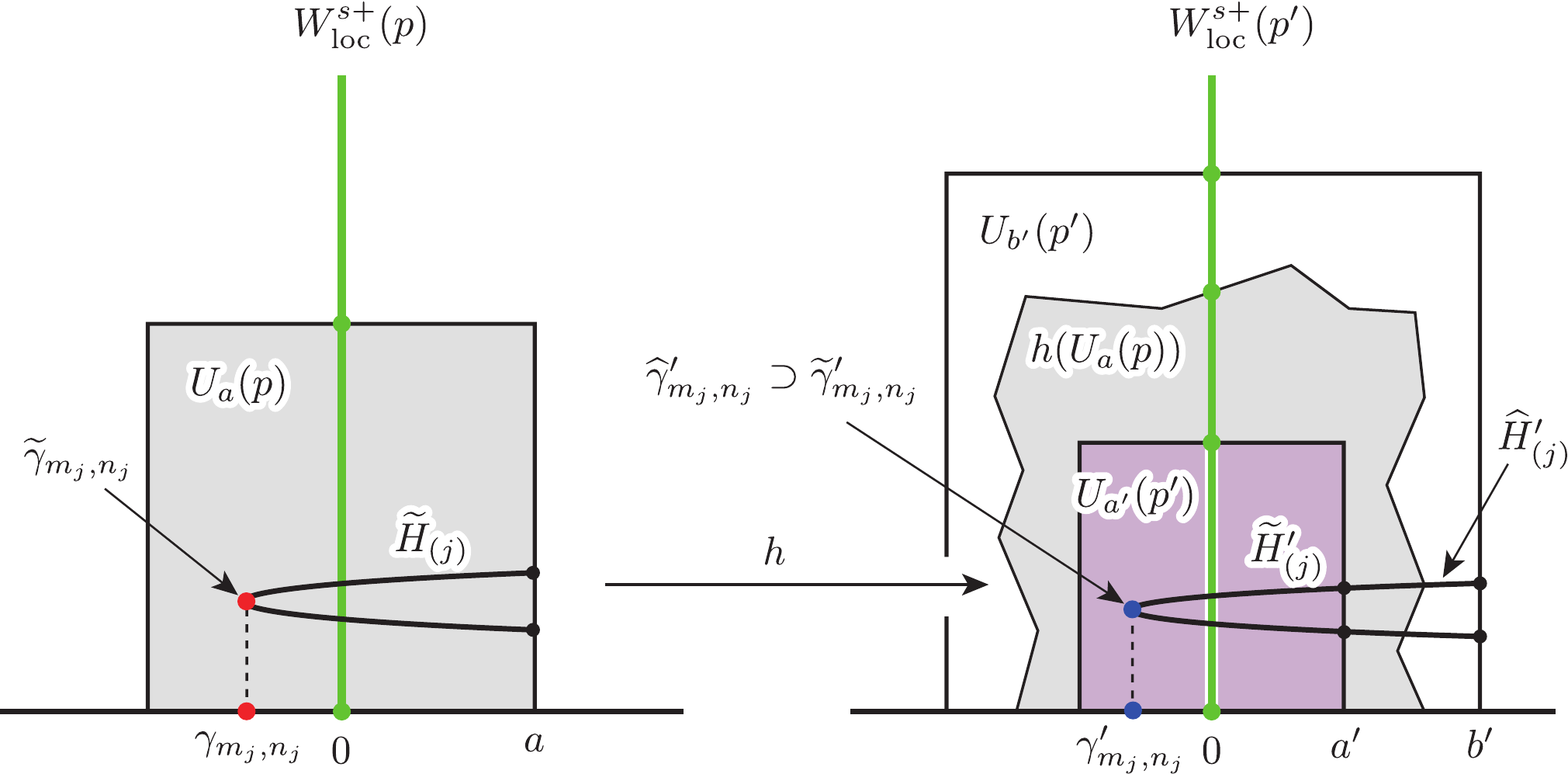}}
\caption{The image $h(\widetilde H_{(j)})$ is contained in $\widehat H_{(j)}'$, 
but $h(\widetilde H_{(j)}^\pm)$ is not necessarily contained in $\widehat H_{(j)}'^\pm$.}
\label{fig_UaUb}
\end{figure}

Since $h$ is only supposed to be a homeomorphism, $h(\widetilde\gamma_{m,n})\cap U_{a'}(p')$  
would not be equal to $\widetilde\gamma_{m,n}'$.
We will show that this equality holds in the limit.
For the sequences $\{m_j\}$, $\{n_j\}$ given in Section \ref{S_limit}, 
we set $\widetilde H_{m_j,n_j}=\widetilde H_{(j)}$, $H_{m_j,n_j}=H_{(j)}$, $\widetilde H_{m_j,n_j}'=\widetilde H_{(j)}'$ and $H_{m_j,n_j}'=H_{(j)}'$ for simplicity.
Similarly, suppose that $\widehat H_{(j)}'$ is the component of $W^u(p')\cap U_{b'}(p')$ containing $\widetilde H_{(j)}'$ and $\widehat\gamma_{m_j,n_1}'$ is the front curve of $\widehat H_{(j)}'$. 
The distance between $\boldsymbol{x}$, $\boldsymbol{y}$ in $U_a(p)$ is denoted by $d(\boldsymbol{x},\boldsymbol{y})$ and 
that between $\boldsymbol{x}'$, $\boldsymbol{y}'$ in $U_{a'}(p')$ by $d'(\boldsymbol{x}',\boldsymbol{y}')$.

The \emph{path metric} on $\widetilde H_{(j)}$ is denoted by $d_{\widetilde H_{(j)}}$.
That is, for any $\boldsymbol{x}$, $\boldsymbol{y}\in \widetilde H_{(j)}$, $d_{\widetilde H_{(j)}}(\boldsymbol{x},\boldsymbol{y})$ is the length of a shortest piecewise smooth curve in $\widetilde H_{(j)}$ 
connecting $\boldsymbol{x}$ with $\boldsymbol{y}$.
The path metrics $d_{\widetilde H_{(j)}'}$ on $\widetilde H_{(j)}'$ and $d_{\widehat H_{(j)}'}$ on $\widehat H_{(j)}'$ 
are defined similarly.

\begin{lemma}\label{l_d_H}
\begin{enumerate}[\rm (i)]
\item
For any $\varepsilon>0$, there exists a constant $\eta(\varepsilon)>0$ independent of $j\in\mathbb{N}$ and satisfying the 
following conditions.
\begin{itemize}
\item
$\lim_{\varepsilon\to 0}\eta(\varepsilon)=0$.
\item
Let $\boldsymbol{x}$, $\boldsymbol{y}$ be any points of $\widetilde H_{(j)}$ both of which are contained in one 
of $\widetilde H_{(j)}^+$ and $\widetilde H_{(j)}^-$.
If $d(\boldsymbol{x},\boldsymbol{y})<\eta(\varepsilon)$, then $d_{\widetilde H_{(j)}}(\boldsymbol{x},\boldsymbol{y})<\varepsilon$.
\end{itemize}
\item
For any $\varepsilon>0$, there exists a constant $\delta(\varepsilon)>0$ independent of $j\in\mathbb{N}$ and satisfying the 
following conditions.
\begin{itemize}
\item
$\lim_{\varepsilon\to 0}\delta(\varepsilon)=0$.
\item
Let $\boldsymbol{x}$, $\boldsymbol{y}$ be any points of $\widetilde H_{(j)}$ both of which are contained in one of $\widetilde H_{(j)}^+$ and $\widetilde H_{(j)}^-$.
If $d_{\widetilde H_{(j)}}(\boldsymbol{x},\boldsymbol{y})<\delta(\varepsilon)$ and $\boldsymbol{x}'=h(\boldsymbol{x})$ and $\boldsymbol{y}'=h(\boldsymbol{y})$ are contained in 
$\widetilde H_{(j)}'$, then $d_{\widetilde H_{(j)}'}(\boldsymbol{x}',\boldsymbol{y}')<\varepsilon$.
\end{itemize}
\end{enumerate}
\end{lemma}

One can take these constants $\eta(\varepsilon)$, $\delta(\varepsilon)$ so that they work also for 
$d_{\widetilde H_{(j)}'}$ and $d_{\widehat H_{(j)}'}$.

\begin{proof}
(i)
The assertion is proved immediately from the fact that 
$\widetilde H_{(j)}^\pm$ uniformly converges to a disk $H^\natural$ in $D_a(p)$ 
such that $d(\boldsymbol{x},\boldsymbol{y})=d_{H^\natural}(\boldsymbol{x},\boldsymbol{y})$ for any $\boldsymbol{x},\boldsymbol{y}\in H^\natural$.
\smallskip

\noindent(ii)
Suppose that $\boldsymbol{x},\boldsymbol{y}\in \widetilde H_{(j)}^+$.
First we consider the case that both $\boldsymbol{x}'$ and $\boldsymbol{y}'$ are contained in one of $\widetilde H_{(j)}'^+$ and $\widetilde H_{(j)}'^-$, say $\widetilde H_{(j)}'^+$.
If $d_{\widetilde H_{(j)}'}(\boldsymbol{x}',\boldsymbol{y}')\geq \varepsilon$, then it follows from  the assertion (i) that $d'(\boldsymbol{x}',\boldsymbol{y}')\geq \eta(\varepsilon)$.
Since $h$ is uniformly continuous on $U_a(p)$, there exists a constant $\delta_1(\varepsilon)>0$ 
with $\lim_{\varepsilon\to 0}\delta_1(\varepsilon)=0$ and $d(\boldsymbol{x},\boldsymbol{y})\geq \delta_1(\varepsilon)$.
Hence, in particular, $d_{\widetilde H_{(j)}}(\boldsymbol{x},\boldsymbol{y})\geq \delta_1(\varepsilon)$.
Thus $d_{\widetilde H_{(j)}}(\boldsymbol{x},\boldsymbol{y})< \delta_1(\varepsilon)$ implies $d_{\widetilde H_{(j)}'}(\boldsymbol{x}',\boldsymbol{y}')<\varepsilon$.

Next we suppose that $\boldsymbol{x}'\in \widetilde H_{(j)}'^+$ and $\boldsymbol{y}'\in \widetilde H_{(j)}'^-$.
Consider a shortest curve $\alpha$ in $\widetilde H_{(j)}$ connecting $\boldsymbol{x}$ and $\boldsymbol{y}$.
Since $\alpha'=h(\alpha)$ is contained in $\widehat H_{(j)}'$, $\alpha'$ intersects $\widehat\gamma_{m_j,n_j}'$ non-trivially.
Let $\boldsymbol{z}$ be one of the intersection points of $\alpha$ with $h^{-1}(\widehat\gamma_{m_j,n_j}')$.
See Figure \ref{fig_d_H}.
\begin{figure}[hbt]
\centering
\scalebox{0.6}{\includegraphics[clip]{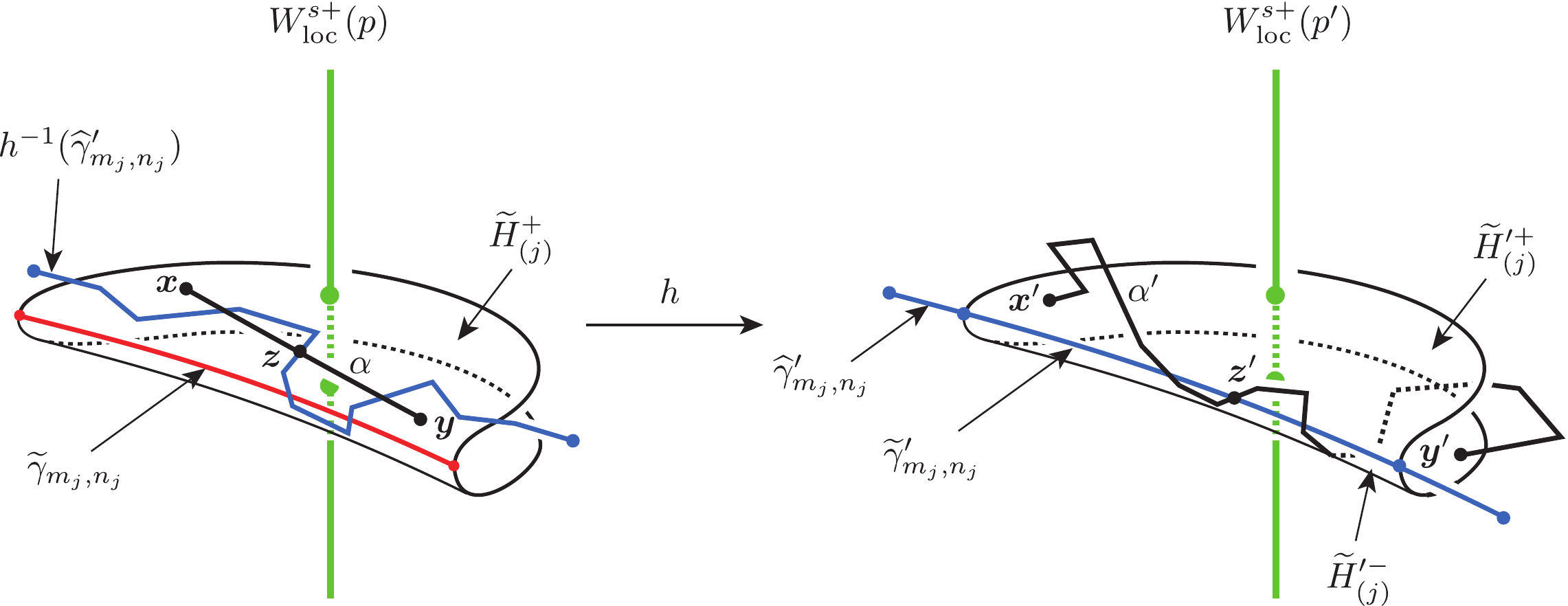}}
\caption{The case of $\boldsymbol{x},\boldsymbol{y}\in \widetilde H_{(j)}^+$, 
$\boldsymbol{x}'\in \widetilde H_{(j)}'^+$ and $\boldsymbol{y}'\in \widetilde H_{(j)}'^-$.}
\label{fig_d_H}
\end{figure}
Suppose that $d_{\widetilde H_{(j)}}(\boldsymbol{x},\boldsymbol{y})<\delta_1(\varepsilon/2)$.
Since $d_{\widetilde H_{(j)}}(\boldsymbol{x},\boldsymbol{y})=d_{\widetilde H_{(j)}}(\boldsymbol{x},\boldsymbol{z})+d_{\widetilde H_{(j)}}(\boldsymbol{z},\boldsymbol{y})$, 
$$d_{\widetilde H_{(j)}}(\boldsymbol{x},\boldsymbol{z})<\delta_1(\varepsilon/2)\quad\text{and}\quad d_{\widetilde H_{(j)}}(\boldsymbol{z},\boldsymbol{y})<\delta_1(\varepsilon/2).$$
Since $\boldsymbol{x}',\boldsymbol{z}'\in \widehat H_{(j)}'^+$ and $\boldsymbol{z}',\boldsymbol{y}'\in \widehat H_{(j)}'^-$, by the result in the previous case 
we have $d_{\widehat H_{(j)}'}(\boldsymbol{x}',\boldsymbol{z}')<\varepsilon/2$ and $d_{\widehat H_{(j)}'}(\boldsymbol{z}',\boldsymbol{y}')<\varepsilon/2$, and hence 
$$d_{\widetilde H_{(j)}'}(\boldsymbol{x}',\boldsymbol{y}')=d_{\widehat H_{(j)}'}(\boldsymbol{x}',\boldsymbol{y}')<\varepsilon.$$
Thus $\delta(\varepsilon):=\delta_1(\varepsilon/2)$ satisfies the conditions of (ii).
\end{proof}

The following result is a key of this paper.

\begin{lemma}\label{l_HH}
For any $\varepsilon>0$, there exists $j_0\in\mathbb{N}$ such that, for any $j\geq j_0$,  
$$h(\widetilde\gamma_{m_j,n_j})\cap \widetilde H_{(j)}'\subset \mathcal{N}_\varepsilon(\widetilde\gamma_{m_j,n_j}',\widetilde H_{(j)}'),$$
where $\mathcal{N}_\varepsilon(\widetilde\gamma_{m_j,n_j}',\widetilde H_{(j)}')$ is the $\varepsilon$-neighborhood of 
$\widetilde\gamma_{m_j,n_j}'$ in $\widetilde H_{(j)}'$.
\end{lemma}

Figure \ref{fig_HH} illustrates the situation of Lemma \ref{l_HH}.
\begin{figure}[hbt]
\centering
\scalebox{0.6}{\includegraphics[clip]{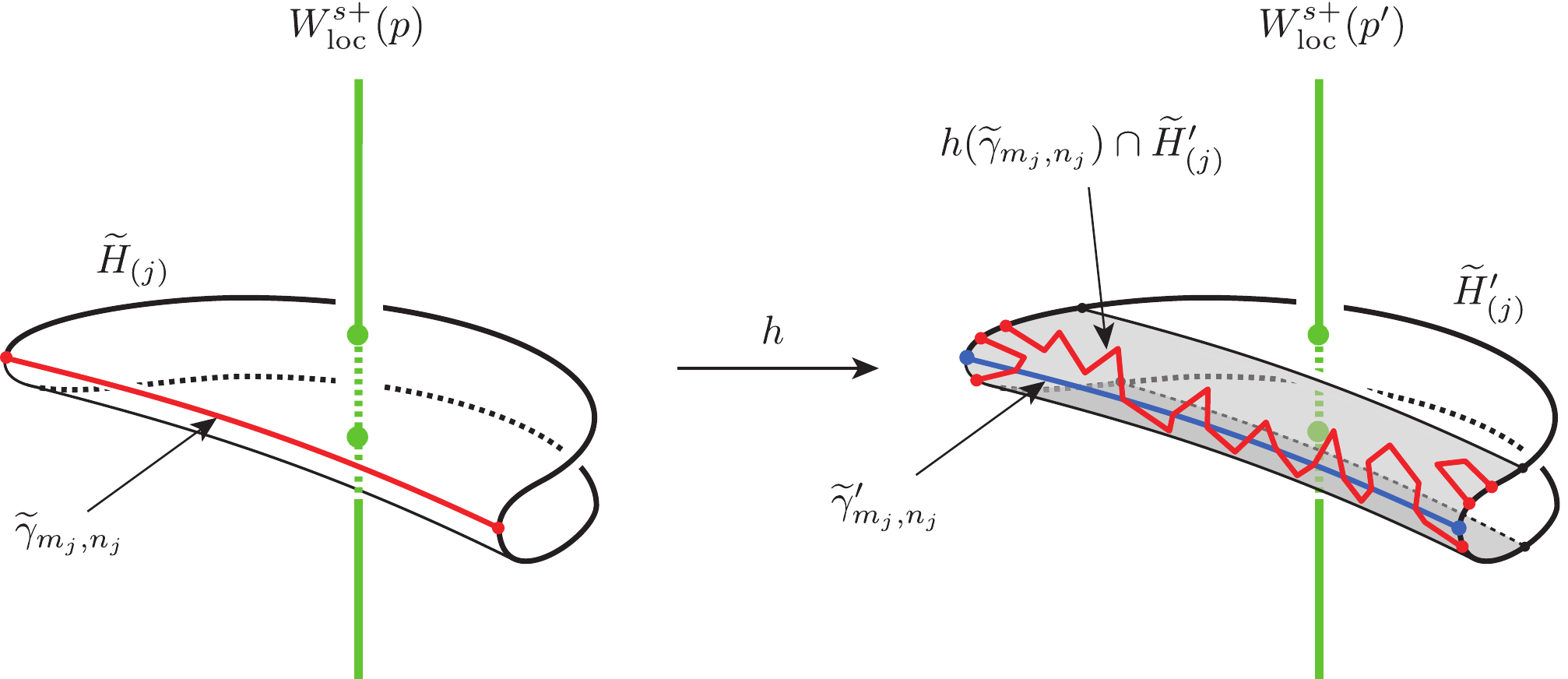}}
\caption{The shaded region represents $\mathcal{N}_\varepsilon(\widetilde\gamma_{m_j,n_j}',\widetilde H_{(j)}')$.}
\label{fig_HH}
\end{figure}

\begin{proof}
For $\sigma=\pm$, we will show that
$h^{-1}(\widetilde H_{(j)}'^{\sigma}\setminus \mathcal{N}_\varepsilon(\gamma_{m_j,n_j}',\widetilde H_{(j)}'))\subset \widetilde H_{(j)}^{\sigma}$ 
for all sufficiently large $j$.
Since $h^{-1}|_{U_{a'}(p')}$ is uniformly continuous, there exists $\nu(\varepsilon)>0$ such that, 
for any $\boldsymbol{x}',\boldsymbol{y}'\in U_{a'}(p')$ with $d'(\boldsymbol{x}',\boldsymbol{y}')<\nu(\varepsilon)$, the inequality $d(\boldsymbol{x},\boldsymbol{y})<\eta(\delta(\varepsilon))$ 
holds, 
where $\boldsymbol{x}=h^{-1}(\boldsymbol{x}')$, $\boldsymbol{y}=h^{-1}(\boldsymbol{y}')$.
Since both $\widetilde H_{(j)}'^+$ and $\widetilde H_{(j)}'^{-}$ uniformly converge to the same half disk $H'^\natural$ 
in $D_{a'}(p')$, there exists $j_0\in \mathbb{N}$ such that, for any $j\geq j_0$ and any $\boldsymbol{x}' \in \widetilde H_{(j)}'^\sigma\setminus \mathcal{N}_\varepsilon(\widetilde\gamma_{(j)}',\widetilde H_{(j)}')$, $d'(\boldsymbol{x}',\boldsymbol{y}')$ is less than 
$\nu(\varepsilon)$, where 
$\boldsymbol{y}'$ is the element of $\widetilde H_{(j)}'^{-\sigma}$ with $\mathrm{pr}(\boldsymbol{x}')=\mathrm{pr}(\boldsymbol{y}')$.
Then we have $d(\boldsymbol{x},\boldsymbol{y})<\eta(\delta(\varepsilon))$.
If both $\boldsymbol{x}$ and $\boldsymbol{y}$ were contained in one of $\widetilde H_{(j)}^\sigma$ and $\widetilde H_{(j)}^{-\sigma}$, then 
by Lemma \ref{l_d_H}\,(i) $d_{\widetilde H_{(j)}}(\boldsymbol{x},\boldsymbol{y})<\delta(\varepsilon)$.
Then, by Lemma \ref{l_d_H}\,(ii), $d_{\widetilde H_{(j)}'}(\boldsymbol{x}',\boldsymbol{y}')$ 
would be less than $\varepsilon$.
This contradicts that  $\boldsymbol{x}' \in \widetilde H_{(j)}'^\sigma\setminus \mathcal{N}_\varepsilon(\widetilde\gamma_{m_j,n_j}',\widetilde H_{(j)}')$ and 
$\boldsymbol{y}'\in \widetilde H_{(j)}'^{-\sigma}$.
See Figure \ref{fig_HH2}.
\begin{figure}[hbt]
\centering
\scalebox{0.6}{\includegraphics[clip]{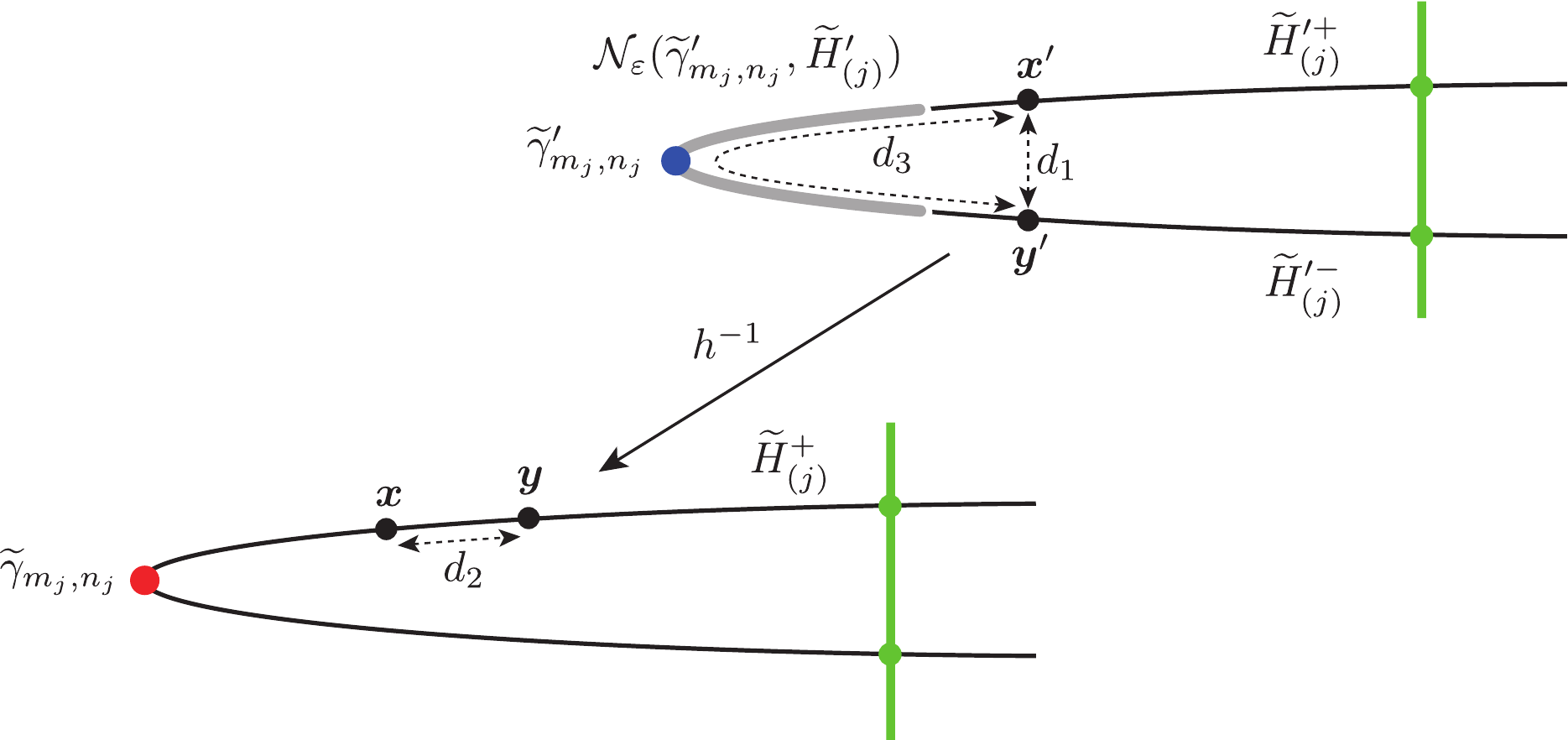}}
\caption{The situation which does not actually occur.
$d_1:=\mathrm{dist}(\boldsymbol{x}',\boldsymbol{y}')<\nu(\varepsilon)$, 
$d_2:=\mathrm{dist}_{\widetilde H_{(j)}}(\boldsymbol{x},\boldsymbol{y})<\delta(\varepsilon)$ and 
$d_3:=\mathrm{dist}_{\widetilde H_{(j)}'}(\boldsymbol{x}',\boldsymbol{y}')<\varepsilon$.}
\label{fig_HH2}
\end{figure}
Thus, if $\boldsymbol{y}$ is contained in $\widetilde H_{(j)}^\sigma$, then $\boldsymbol{x}$ is not in $\widetilde H_{(j)}^\sigma$.
In particular, $\boldsymbol{x}$ is not contained in $\widetilde\gamma_{m_j,n_j}=\widetilde H_{(j)}^+\cap \widetilde H_{(j)}^-$, and so $\widetilde\gamma_{m_j,n_j}\cap h^{-1}(\widetilde H_{(j)}'^\sigma\setminus \mathcal{N}_\varepsilon(\widetilde\gamma_{m,n}',\widetilde H_{(j)}'))=\emptyset$.
Since $h^{-1}(\widetilde H_{(j)}'^\sigma\setminus \mathcal{N}_\varepsilon(\widetilde\gamma_{m,n}',\widetilde H_{(j)}'))$ is 
connected, it follows that $h^{-1}(\widetilde H_{(j)}'^\sigma\setminus \mathcal{N}_\varepsilon(\widetilde\gamma_{m,n}',\widetilde H_{(j)}'))\subset \widetilde H_{(j)}^\sigma$ for $\sigma=\pm$, and hence 
 $h^{-1}(\mathcal{N}_\varepsilon(\widetilde\gamma_{m_j,n_j}',\widetilde H_{(j)}'))\supset
\widetilde\gamma_{m_j,n_j}\cap h^{-1}(\widetilde H_{(j)}')$.
This completes the proof.
\end{proof}

From the proof of Lemma \ref{l_HH}, we know that there exists a simple curve in $h(\widetilde\gamma_{m_j,n_j})\cap 
\widetilde H_{(j)}'$ 
connecting the two components of $\partial \widetilde H_{(j)}'\cap \partial \mathcal{N}_\varepsilon(\widetilde\gamma_{m_j,n_j}',\widetilde H_{(j)}')$.
The following corollary says that the images of certain straight segments in $D_a(p)$ by the homeomorphism $h$ 
are naturally straight segments in $D_{a'}(p')$.

\begin{cor}\label{c_HH}
For the limit straight segment $\gamma_0^\natural$ of $\gamma_{m_j,n_j}$, 
$h(\gamma_0^\natural)\cap D_{a'}(p')$ is the limit straight segment of $\gamma_{m_j,n_j}'$, i.e.\ $h(\gamma_0^\natural)\cap D_{a'}(p')=\gamma_0'^{\,\natural}$.
\end{cor}

\begin{proof}
Since $\gamma_0^\natural$ is the limit straight segment of $\widetilde \gamma_{m_j,n_j}$ and $h$ is uniformity 
continuous, 
$h(\gamma^\natural_0)\cap D_{a'}(p')$ is the limit of $h(\widetilde\gamma_{m_j,n_j})\cap \widetilde H_{(j)}'$.
It follows from Lemma \ref{l_HH} that $h(\gamma_0^\natural)\cap D_{a'}(p')$ is also the limit of $\mathrm{pr}(\widetilde\gamma_{m_j,n_j}')=\gamma_{m_j,n_j}'$, that is, $h(\gamma_0^\natural)\cap D_{a'}(p')$ is 
equal to the limit straight segment of $\gamma_{m_j,n_j}'$.
\end{proof}

For any straight segment $l$ in $D_a(p)$ such that $h(l)$ is also a straight segment in 
$D_{b'}(p')$, we denote $h(l)\cap D_{a'}(p')$ simply by $h(l)$. 
In particular, Corollary \ref{c_HH} implies that $h(\gamma_0^\natural)=\gamma_0'^{\,\natural}$.

\section{Proof of Theorem \ref{thm_A}}\label{S_pA}

Suppose that $\mathrm{St}_a(p)$ is the set of oriented proper straight segments in $D_a(p)$ passing 
through $0$, 
that is, each element of $\mathrm{St}_a(p)$ is an oriented diameter of the disk $D_a(p)$.
For any $l\in \mathrm{St}_a(p)$ and $n\in \mathbb{N}$, the component of $f^n(l)\cap U_a(p)$ containing $0$ is 
also an element of $\mathrm{St}_a(p)$.
We denote the element simply by $f^n(l)$.

Since $f^n|_{D_a(p)}$ preserves angles on $D_a(p)$, by \eqref{eqn_theta}, for any $k,n\in\mathbb{N}$, 
$$\vartheta(\gamma_{m,n})-\vartheta(\gamma_{m+k,n})=\vartheta(\gamma_{m,0})-\vartheta(\gamma_{m+k,0})
\to \vartheta(\gamma_0)-\vartheta(\gamma_0)=0$$
as $m\to \infty$.
Moreover it follows from \eqref{eqn_dmn} that $\lim_{j\to\infty}d_{m_j+k,n_j}=w_0\lambda^k$.
By these facts together with Lemma \ref{l_gamma_star}, one can show that 
$\gamma_{m_j+k,n_j}$ uniformly converges as $m\to \infty$ to a straight segment $\gamma_k^\natural$ in $U_a(p)$ 
with 
\begin{equation}\label{eqn_zeta_k}
\vartheta(\gamma_k^\natural)=\theta^\natural\quad\text{and}\quad 
d(0,\gamma_k^\natural)=w_0\lambda^k.
\end{equation}
Thus we have obtained the parallel family $\{\gamma_k^\natural\}$ of oriented straight segments in $D_a(p)$.
See Figure \ref{fig_parallel}.
\begin{figure}[hbt]
\centering
\scalebox{0.6}{\includegraphics[clip]{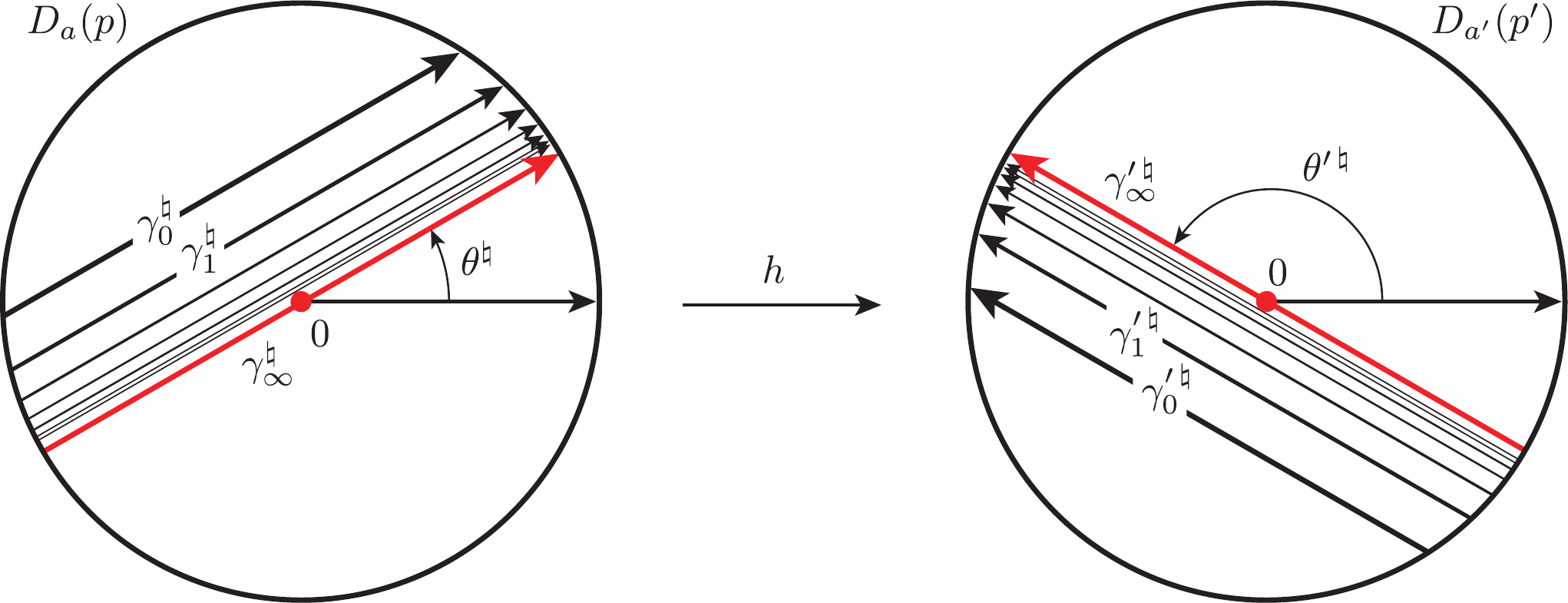}}
\caption{The images of the parallel straight segments $\gamma_k^\natural$ in $D_a(p)$ by $h$.}
\label{fig_parallel}
\end{figure}
By Corollary \ref{c_HH}, $\{\gamma_k'^{\,\natural}\}$ with $\gamma_k'^{\,\natural}=h(\gamma_k^\natural)$ 
is also a parallel family of oriented straight segments in $D_{a'}(p')$.
Since $\gamma_k'^\natural$ is the limit of $\gamma_{m_j+k,n_j}'$ as $j\to\infty$, we have the equations 
\begin{equation}\label{eqn_zeta_k'}
\vartheta(\gamma_k'^{\,\natural})=\theta'^{\,\natural}\quad\text{and}\quad 
d(0,\gamma_k'^{\,\natural})=w_0'\lambda'^k.
\end{equation}
corresponding to \eqref{eqn_zeta_k} for some $\theta'^{\,\natural}$ and $w_0'>0$.
Let $\gamma_\infty^\natural\in \mathrm{St}_a(p)$ 
(resp.\ $\gamma_\infty'^{\,\natural}\in \mathrm{St}_{a'}(p')$) be the limit of $\gamma_k^\natural$ 
(resp.\ $\gamma_k'^{\,\natural}$).

\begin{proof}[Proof of Theorem \ref{thm_A}]
By Lemma \ref{l_gamma_star} and \eqref{eqn_dmn}, $w_0=\lim_{j\to \infty}\widetilde d_0\widetilde t_0\lambda^{m_j}r^{n_j}$.
This implies that 
$$\lim_{j\to \infty}\left(\frac{m_j}{n_j}\log \lambda+\log r\right)=\lim_{j\to\infty}\frac1{n_j}\log 
\frac{w_0}{\widetilde d_0\widetilde t_0}=0$$
and hence $\lim_{j\to\infty}\dfrac{m_j}{n_j}=-\dfrac{\log r}{\log \lambda}$.
Applying the same argument to $\gamma_{m_j,n_j}'^{\,\natural}$, we also have 
$\lim_{j\to\infty}\dfrac{m_j}{n_j}=-\dfrac{\log r'}{\log \lambda'}$.
This shows the part \eqref{A1} of Theorem \ref{thm_A}.

Now we will prove the part \eqref{A2}.
For any $n\in \mathbb{N}\cup \{0\}$, we set $f^n(\gamma_\infty^\natural)=\gamma_{\infty,n}^\natural$ and 
$f'^n(\gamma_\infty'^{\,\natural})=\gamma_{\infty,n}'^{\,\natural}$.
By Corollary \ref{c_HH}, 
\begin{equation}\label{eqn_h_gamma}
h(\gamma_{\infty,n}^\natural)=h(f^n(\gamma_\infty^\natural))=f'^n(h(\gamma_\infty^\natural))=f'^n(\gamma_\infty'^{\,\natural})=
\gamma_{\infty,n}'^{\,\natural}.
\end{equation}
We identify $\mathrm{St}_a(p)$ with the unit circle $S^1=\{z\in\mathbb{C}\,;\,|z|=1\}$ by corresponding $l\in \mathrm{St}_a(p)$ to
$e^{\sqrt{-1}\vartheta(l)}$.
Then the action of $f$ on $\mathrm{St}_a(p)$ is equal to the $\theta$-rotation $R_\theta$ on $S^1$ defined by 
$R_\theta(z)=e^{\sqrt{-1}\theta}z$.

If $\theta/2\pi=v/u$ for coprime positive integers $u$, $v$ with $0\leq v<u$.
Since $h(\gamma_\infty^\natural)=\gamma_\infty'^{\,\natural}$, we have 
$f'^k(\gamma_\infty'^{\,\natural})\neq \gamma_\infty'^{\,\natural}$ for $k=1,\dots,u-1$ and 
$f'^u(\gamma_\infty'^{\,\natural})=\gamma_\infty'^{\,\natural}$.
This implies that $\theta'/2\pi=v'/u$ for some $v'\in\mathbb{N}$ with $0\leq v'<u$.
Since $h|_{D_a(p)}:D_a(p)\to D_{a'}(p')$ is a homeomorphism with the correspondence  
$h(R_\theta^k(\gamma_\infty^\natural))=R_{\theta'}^k(\gamma_\infty'^\natural)$ $(k=0,1,\dots,u-1)$, 
there exists an orientation-preserving homeomorphism $\eta_0:S^1\to S^1$ with 
$\eta_0(e^{\sqrt{-1}(\theta^\natural+k\theta)})=e^{\sqrt{-1}(\theta'^\natural+k\theta')}$ 
for $k=0,1,\dots,u-1$.
We set $\Gamma=\bigl\{e^{\sqrt{-1}(\theta^\natural+k\theta)};\,k=0,1,\dots,u-1\bigr\}$ and  $\Gamma'=\bigl\{e^{\sqrt{-1}(\theta'^\natural+k\theta')};\,k=0,1,\dots,u-1\bigr\}$. 
Then $\bigl[e^{\sqrt{-1}\theta^\natural},e^{\sqrt{-1}(\theta^\natural+\theta)}\bigr)\cap \Gamma$ consists of $v$ points, 
where $[a,b)$ denotes the positively oriented half-open interval in $S^1$ for $a,b\in S^1$ with $a\neq b$.
Since moreover $\eta_0\bigl(\bigl[e^{\sqrt{-1}\theta^\natural},e^{\sqrt{-1}(\theta^\natural+\theta)}\bigr)\cap \Gamma\bigr)
=\bigl[e^{\sqrt{-1}\theta'^\natural},e^{\sqrt{-1}(\theta'^\natural+\theta')}\bigr)\cap \Gamma'$ consists of $v'$ points, it follows that $v=v'$, and hence $\theta=\theta'$.

Next we suppose that $\theta/2\pi$ is irrational.
Then, for any $l\in \mathrm{St}_a(p)$, there exists a subsequence $\{n_k\}$ of $\mathbb{N}$ such that 
the sequence $\gamma_{\infty,n_k}^\natural$ uniformly converges to $l$ as $k\to \infty$.
By \eqref{eqn_h_gamma}, $\gamma_{\infty,n_k}'^{\,\natural}$ uniformly converges to $l'=h(l)$.
Since $\gamma_{\infty,n_k}'^{\,\natural}\in \mathrm{St}_{a'}(p')$, 
$l'$ is also an element of $\mathrm{St}_{a'}(p')$.
Thus we have a homeomorphism $\eta:S^1\to S^1$ with respect to which $R_\theta$ and $R_{\theta'}$ 
are conjugate.
Since the rotation number is invariant under topological conjugations, $\theta/2\pi=\theta'/2\pi\mod 1$ 
holds.
This completes the proof of the part \eqref{A2}.
\end{proof}

\section{Proof of Theorem \ref{thm_B}}\label{S_PB}

In this section, we will prove Theorem \ref{thm_B}.
Suppose that $f,f'$ are elements of $\mathrm{Diff}^r(M)$ satisfying the conditions of Theorems \ref{thm_A} 
and $\theta/2\pi$ is irrational.

Since $\theta=\theta'\mod 2\pi$, for any $k,j\in\mathbb{N}$, 
\begin{equation}\label{eqn_theta_gamma}
\vartheta(\gamma_{\infty,k}^\natural)-\vartheta(\gamma_{\infty,j}^\natural)=\vartheta(\gamma_{\infty,k}'^{\,\natural})-\vartheta(\gamma_{\infty,j}'^{\,\natural})=(k-j)\theta\mod 2\pi.
\end{equation}
Let $l_j$ $(j=1,2)$ be any elements of $\mathrm{St}_a(p)$.
As in the proof of Theorem \ref{thm_A}, there exist subsequences $\{n_k\}$, $\{n_j\}$ of $\mathbb{N}$ such that 
the sequencers $\{\gamma_{\infty,n_k}^\natural\}$, $\{\gamma_{\infty,n_j}^\natural\}$ uniformly 
converge to $l_1$ and $l_2$ respectively.
Then, $\{\gamma_{\infty,n_k}'^{\,\natural}\}$, $\{\gamma_{\infty,n_j}'^{\,\natural}\}$ also uniformly 
converge to the elements $l_1'=h(l_1)$ and $l_2'=h(l_2)$ of $\mathrm{St}_{a'}(p')$ respectively.
Then, by \eqref{eqn_theta_gamma},  
\begin{equation}\label{eqn_theta_l}
\vartheta(l_2)-\vartheta(l_1)=\vartheta(l_2')-\vartheta(l_1')\mod 2\pi.
\end{equation}

For the proof of Theorem \ref{thm_B}, we need another family of straight segments in $D_a(p)$.
Fix an integer $a_0$ with 
$$a_0>\max\left\{\frac{\log(2r)}{\log (\lambda^{-1})},\frac{\log(2r')}{\log (\lambda'^{-1})}\right\}.$$
For any $k\geq 0$, we consider the straight segment $\xi_k^\natural=f^k(\gamma_{a_0k}^\natural)\cap D_a(p)$.
By \eqref{eqn_zeta_k},  
\begin{equation}\label{eqn_theta_xi}
\vartheta(\xi_k^\natural)-\vartheta(\xi_0^\natural)=k\theta\mod 2\pi
\quad\text{and}\quad d(0,\xi_k^\natural)=w_0\lambda^{a_0k}r^k<2^{-k}w_0.
\end{equation}
Similarly, by \eqref{eqn_zeta_k'}, $\xi_k'^{\,\natural}=h(\xi_k^\natural)$ is a straight segment in $D_{a'}(p')$ with
\begin{equation}\label{eqn_theta_xi'}
\vartheta(\xi_k'^{\,\natural})-\vartheta(\xi_0'^{\,\natural})=k\theta\mod 2\pi
\quad\text{and}\quad d(0,\xi_k'^{\,\natural})=w_0'\lambda'^{a_0k}r'^k<2^{-k}w_0'.
\end{equation}

\begin{proof}[Proof of Theorem \ref{thm_B}]
Let $\alpha$ be the element of $\mathrm{St}_a(p)$ with $\vartheta(\xi_0^\natural)-\vartheta(\alpha)=\pi/2$ 
and $\alpha'=h(\alpha)\in \mathrm{St}_{a'}(p')$.
We will show that $\theta_{\alpha'}:=\vartheta(\xi_0'^{\,\natural})-\vartheta(\alpha')$ 
is also equal to $\pi/2\mod 2\pi$.
See Figure \ref{fig_beta}.
\begin{figure}[hbt]
\centering
\scalebox{0.6}{\includegraphics[clip]{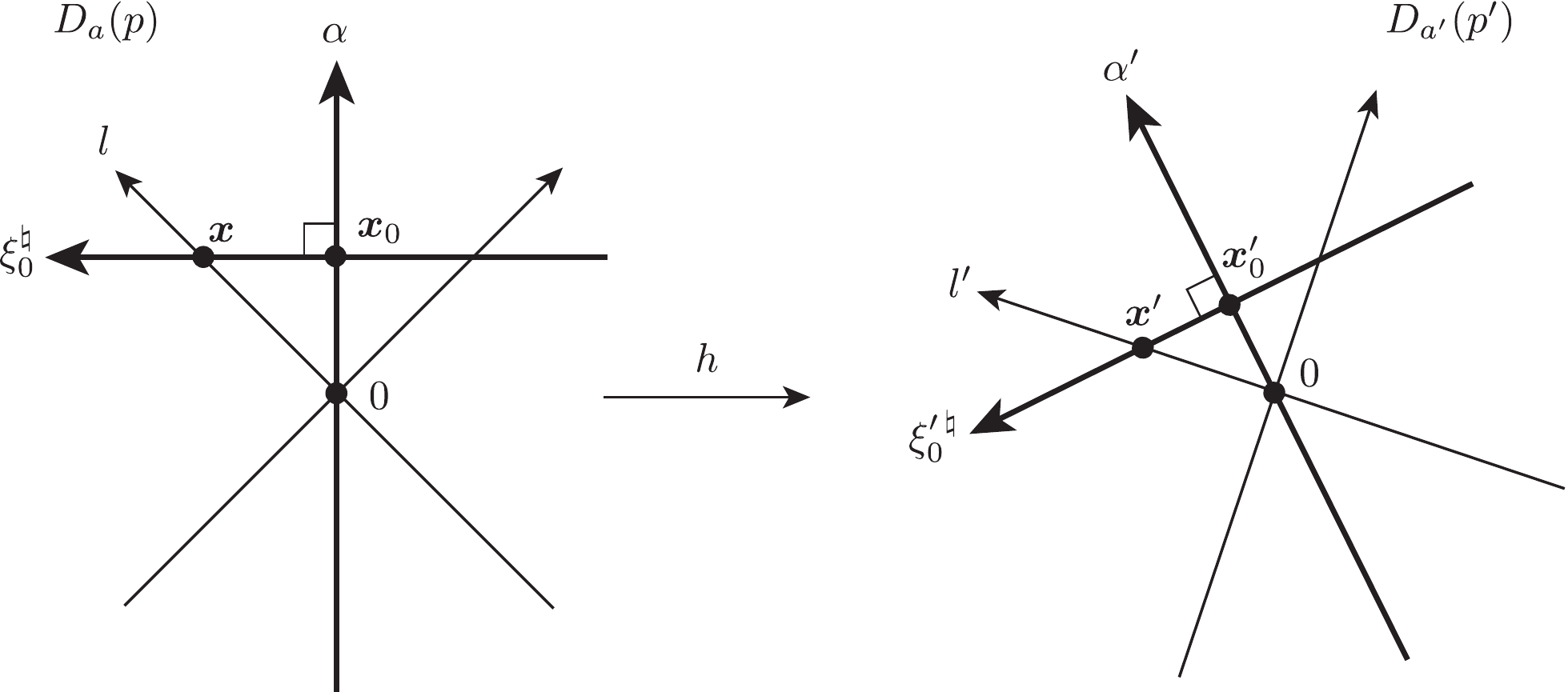}}
\caption{Correspondence of straight segments via $h$.}
\label{fig_beta}
\end{figure}
In fact, since $\theta/2\pi$ is irrational, by \eqref{eqn_theta_xi} there exists a subsequence $\xi_{k_j}^\natural$ uniformly converges to $\alpha$.
Since $h|_{D_a(p)}$ is uniformly continuous, $\xi_{k_j}'^{\,\natural}$ also uniformly converges to $\alpha'$.
On the other hand, since $\vartheta(\xi_{k_j}^\natural)-\vartheta(\alpha)=k_j\theta+\pi/2\mod 2\pi$ and 
$\vartheta(\xi_{k_j}'^{\,\natural})-\vartheta(\alpha')=k_j\theta+\theta_{\alpha'}\mod 2\pi$, 
$$\theta_{\alpha'}-\frac{\pi}2=\bigl(\vartheta(\xi_{k_j}'^{\,\natural})-\vartheta(\alpha')\bigr)
-\bigl(\vartheta(\xi_{k_j}^\natural)-\vartheta(\alpha)\bigr)\to 0\mod 2\pi$$
as $j\to\infty$.
Thus we have $\theta_{\alpha'}=\pi/2\mod 2\pi$.

We denote by $z(\boldsymbol{x})\in \mathbb{C}$ the entry of $\boldsymbol{x}\in D_{a}(p)$ with respect to the linearizing coordinate 
on $D_a(p)$.
Similarly, the entry of $\boldsymbol{x}'\in D_{a'}(p')$ is denoted by $z'(\boldsymbol{x}')$.
Let $\boldsymbol{x}_0$ be the intersection point of $\alpha$ and $\xi_0^\natural$, and let $\boldsymbol{x}_0'=h(\boldsymbol{x}_0)$.
One can set $z(\boldsymbol{x}_0)=\rho_0e^{\sqrt{-1}\omega_0}$ and $z'(\boldsymbol{x}_0')=\rho_0'e^{\sqrt{-1}\omega_0'}$ 
for some $\rho_0>0$, $\rho_0'>0$ and $\omega_0$, $\omega_0'\in \mathbb{R}$.
We define the new linearizing coordinate on $D_{a'}(p')$ by using the linear conformal map such that, for any $\boldsymbol{x}'\in D_{a'}(p')$, 
$z'^{\,\mathrm{new}}(\boldsymbol{x}')=\rho_0\rho_0'^{-1}e^{\sqrt{-1}(\omega_0-\omega_0')}z'(\boldsymbol{x}')$.
Then $z(\boldsymbol{x}_0)=z'^{\,\mathrm{new}}(\boldsymbol{x}_0')$ holds.

For any $\boldsymbol{x}\in \xi_0^\natural$, there exists $l\in \mathrm{St}_a(p)$ with $\{\boldsymbol{x}\}=\xi_0^\natural
\cap l$.
Then $\boldsymbol{x}'=h(\boldsymbol{x})$ is the intersection of $\xi_0'^{\,\natural}$ and $l'=h(l)$.
By \eqref{eqn_theta_l}, 
$\vartheta(l)-\vartheta(\alpha)=\vartheta(l')-\vartheta(\alpha')\mod 2\pi$ 
and hence $z(\boldsymbol{x})=z'^{\,\mathrm{new}}(\boldsymbol{x}')$.
We say the property that $h$ is \emph{identical} on $\xi_0^\natural$.
Since $\theta/2\pi$ is irrational, there exists $k_*\in\mathbb{N}$ satisfying
$$\frac{\pi}3\leq \vartheta(\xi_{k_*}^\natural)-\vartheta(\xi_0^\natural)\leq \frac{\pi}2\mod 2\pi.$$
Then $\xi_{k_*}^\natural$ meets $\xi_0^\natural$ at a single point $\boldsymbol{x}_{k_*}$ 
in $D_a(p)$.
For $\alpha_{k_*}=f^{k_*}(\alpha)$ and $\alpha_{k_*}'=h(\alpha_{k_*})$, 
we have 
$\vartheta(\xi_{k_*}^\natural)-\vartheta(\alpha_{k_*})=\vartheta(\xi_{k_*}'^{\,\natural})-\vartheta(\alpha_{k_*}')=\pi/2$.
Since $h$ is identical at $\boldsymbol{x}_{k_*}$, 
$h$ is proved to be identical on $\xi_{k_*}^\natural$ by an argument as above.
Then one can show inductively that, for any $n\in\mathbb{N}$, $h$ is identical on $\xi_{nk_*}^\natural$.
See Figure \ref{fig_beta_2}.
\begin{figure}[hbt]
\centering
\scalebox{0.6}{\includegraphics[clip]{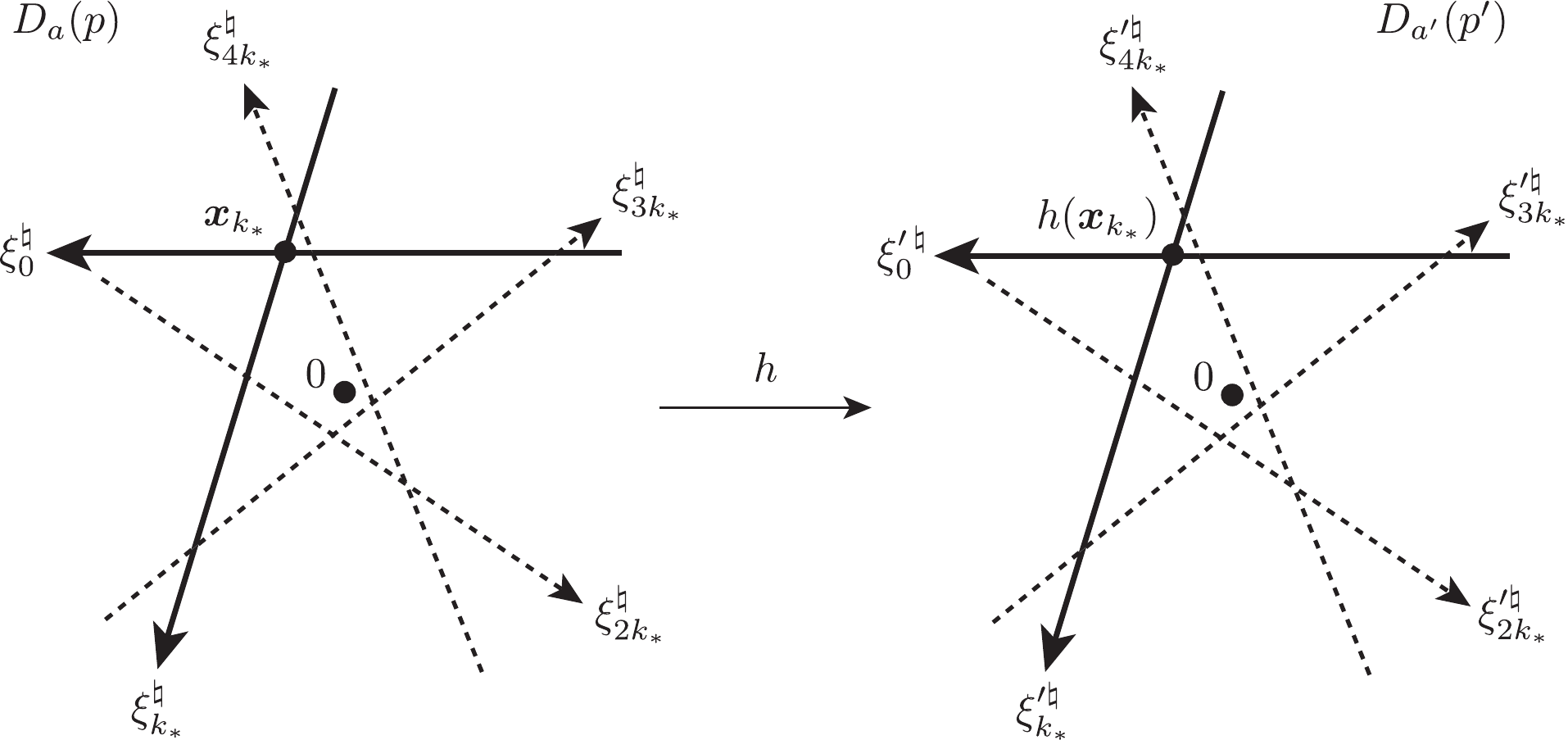}}
\caption{Correspondence via $h$ with respect to the new coordinate on $D_{a'}(p')$.}
\label{fig_beta_2}
\end{figure}
By \eqref{eqn_theta_xi}, $\lim_{n\to\infty}d(0,\xi_{nk_*}^\natural)=0$.
Since moreover $k_*\theta/2\pi$ is irrational, $\overline{\bigcup_{n=1}^\infty \xi_{nk_*}^\natural}$ is equal to $D_a(p)$.
This shows that $h$ is identical on $D_a(p)$.
In particular, this implies that $h|_{D_a(p)}$ is a linear conformal map 
with respect to the original coordinates.
We write $z(q)=\rho_1e^{\sqrt{-1}\omega_1}$ and $z'(q')=\rho_1'e^{\sqrt{-1}\omega_1'}$.
It follows from the assumption of $h(q)=q'$ in our theorems that 
$h(z)=\rho_1'\rho_1^{-1}e^{\sqrt{-1}(\omega_1'-\omega_1)}z$ for any $z\in\mathbb{C}$ with $|z|\leq a$.
In particular, this implies that $h|_{W_{\mathrm{loc}}^u(p)}$ is a linear conformal map.
Let $\widetilde h$ be any other conjugacy homeomorphism between $f$ and $f'$ satisfying the conditions in Theorems \ref{thm_A} and \ref{thm_B}.
In particular, $\widetilde h(p)=p'$ and $\widetilde h(q)=q'$ hold.
Since $z(q)=\rho_1e^{\sqrt{-1}\omega_1}$ and $z'(q')=\rho_1'e^{\sqrt{-1}\omega_1'}$, 
one can show as above that $\widetilde h(z)=\rho_1'\rho_1^{-1}e^{\sqrt{-1}(\omega_1'-\omega_1)}z$ for any $z\in\mathbb{C}$ with $|z|\leq a$ and hence $\widetilde h|_{D_a(p)}=h|_{D_a(p)}$.
This shows the assertion \eqref{B2} of Theorem \ref{thm_B} and $r=r'$. 
Then, by the assertion \eqref{A1} of Theorem \ref{thm_A}, we also have $\lambda=\lambda'$.
This completes the proof.
\end{proof}

Let $\widehat z$ be the homoclinic transverse point of $W^u(p)$ and $W^s(p)$ given in Subsection \ref{ss_bent_disk}.
Fix a sufficiently large $n\in \mathbb{N}$ with $s=f^{-n}(\widehat z)\in D_p(a)$.
Then $s'=h(s)$ is contained in $D_{b'}(p')$.
The following corollary shows that $z(s)/z(q)$ is a modulus for $f$.
Recall that $z(\boldsymbol{x})\in\mathbb{C}$ is the entry of $\boldsymbol{x}$ with respect to 
the complex linearizing coordinate on $D_a(a)$.
The complex number $z'(\boldsymbol{x}')$ is defined similarly for $\boldsymbol{x}'\in D_{a'}(p')$.

\begin{mcorollary}\label{c_C}
Let $f$, $f'$ be elements of $\mathrm{Diff}^r(M)$ satisfying the conditions of Theorems \ref{thm_A} and \ref{thm_B}, 
and let $h$ be a conjugacy homeomorphism between $f$ and $f'$ with $h(p)=p'$ and $h(q)=q'$.
If $h|_{W_{\mathrm{loc}}^u(p)}$ is orientation-preserving, then $z(s)/z(q)=z'(s')/z'(q')$.
Otherwise, $z(s)/z(q)=\overline{z'(s')/z'(q')}$.
\end{mcorollary}

\begin{proof}
Here we only consider the case that $h$ is orientation-preserving.
Since $h|_{D_a(p)}$ is a linear conformal map, 
the triangle with vertices $0,z(q),z(s)$ is similar to that with vertices $0,z'(q'),z'(s')$ 
with respect to Euclidean geometry.
This shows $z(s)/z(q)=z'(s')/z'(q')$.
\end{proof}

\subsection*{Acknowledgement}
The authors would like to thank the referee and the editors for helpful comments and suggestions.


\end{document}